\documentclass[11pt]{article}
\usepackage{amsmath,amsthm,amssymb}
\usepackage{dsfont}
\usepackage[usenames,dvipsnames]{xcolor}
\usepackage{enumerate}
\usepackage{graphicx,hyperref}
\hypersetup{
    linktoc=page,
    linkcolor=red,          
    citecolor=blue,        
    filecolor=blue,      
    urlcolor=cyan,
    colorlinks=true           
}
\usepackage{cite}
\usepackage{comment}
\usepackage{oands}
\usepackage{tikz}
\usepackage{changepage}
\usepackage{bm}
\usepackage{bbm}
\usepackage{mathtools}
\usepackage[margin=1.in]{geometry}
\usepackage{hyperref}
\usepackage[pagewise,mathlines]{lineno}
\usepackage{appendix}
\usepackage{multicol}
\usepackage{microtype}
\usepackage[colorinlistoftodos,textsize=footnotesize]{todonotes}
\usepackage{fancyhdr}

\usepackage[numbers]{natbib}

{\theoremstyle {definition} \newtheorem {defi} {Definition} [section] }
{\theoremstyle {definition} \newtheorem {rmk} {Remark} [defi] }
{\theoremstyle {plain}  \newtheorem {theorem} [defi] {Theorem}}
{\theoremstyle {plain}  \newtheorem {corollary} [defi]{Corollary}}
{\theoremstyle {plain} \newtheorem {prop} [defi]{Proposition}}
{\theoremstyle {plain} \newtheorem {lemma}[defi] {Lemma}}

\newcommand{\expect}[1]{\mathbb{E}\left[#1\right]}
\newcommand{\Hr}{\mathrm H}
\newcommand{\keywords}[1]{\textbf{Keywords and phrases: }#1}
\newcommand{\classification}[1]{\textbf{Mathematics Subject Classification. }#1}

\begin{document}

\title{Unit boundary length quantum disk: \\a study of two different perspectives and their equivalence}
\author{Baptiste Cercl\'e\footnote{baptiste.cercle@ens-lyon.fr. D\'epartement de math\'ematiques de l'ENS de Lyon, 15 parvis Ren\'{e} Descartes, 69342 Lyon, France}}
\date{}

\maketitle
\abstract{The theory of the two-dimensional Liouville Quantum Gravity, first introduced by Polyakov in his 1981 work has become a key notion in the study of random surfaces. In a series of articles, David, Huang, Kupiainen, Rhodes and Vargas, on the one hand, and Duplantier, Miller and Sheffield on the other hand, investigated this topic in the realm of probability theory, and both provided definitions for fundamentals objects of the theory: the unit area quantum sphere and the unit boundary length quantum disk. In a recent article, Aru, Huang and Sun showed that the definitions given in the case of the sphere coincide. We study here the two different perspectives provided for the unit boundary length quantum disk and show that they define the same probabilistic objects by considering two similar limiting procedures giving rise to them.}

\vspace{0.5cm}
\keywords{Liouville Quantum Gravity, Gaussian Free Field, Gaussian Multiplicative Chaos.}

\classification{60D05, 81T20, 81T40.}

\section{Introduction}
\subsection{Motivation and background}

The study of \textit{Liouville Conformal Field Theory} first appeared in Polyakov's seminal article \cite{Pol81} in which the author described a theory of summation over Riemannian metrics based on a path integral approach. This work represents the starting point for the study of the so-called two-dimensional \textit{Quantum Gravity}, which can be roughly described as a model for defining random metrics on surfaces with fixed topology (see \cite{Sie90} for instance). In more physical terms, the approach developed by Polyakov allowed to provide a formulation of non-critical string theory along with a new model for quantizing the space-time evolution of bosonic strings.
Generally speaking, there are several ways to give a meaning to the notion of canonical random surface, but the approaches developed should be in some sense equivalent, the object of study being universal. These perspectives differ on many points, should they concern their means or motivations. We will review here some of them that have become of special interest over the past few years.

To begin with, the Liouville Conformal Field Theory has the special feature that the object of study admits an explicit path integral formulation, defined according to the Liouville action: in this context, a heuristic definition of a canonical random measure on a Riemann surface relies on a generalization of Feynman path integrals to surfaces in the following sense.
Assume that $D$ is a Riemann surface with boundary (possibly empty) $\partial D$, and $g$ any Riemannian metric on this manifold. Given a map $X:\overline{D}\rightarrow\mathbb{R}$, one can define the Liouville action functional on the Riemann surface $S(X,g)$ by the (formal) expression \begin{equation}\label{Liouville_action}
    S(X,g)=\frac{1}{4\pi}\int_D (|\nabla_g X|^2 + Q R_g X + 4\pi\mu e^{\gamma X}) d\lambda_{g} + \int_{\partial D} (Q K_g X + 2\pi\mu_{\partial}e^{\frac{\gamma}{2} X})d\lambda_{\partial g}
\end{equation}
where the geometric terms $R_g$, $K_g$, $d\lambda_{g}$ and $d\lambda_{\partial g}$ are respectively the Ricci scalar curvature, the geodesic curvature, the volume form and the line element in the metric $g$, while the physical constants $\gamma$, $Q$, $\mu$ and $\mu_{\partial}$ correspond to the \textit{Liouville coupling constant}, the \textit{Liouville background charge}, and the \textit{cosmological constants}. In this framework, the law of a canonical random field $\phi$ on $D$ should be described by \begin{equation}\label{law_field}
    \expect{F(\phi)}:=\frac{1}{\mathcal{Z}}\int F(X)e^{-S(X,g)}dX
\end{equation}
where $dX$ would stand for \lq\lq the Lebesgue measure'' on a space of maps $X:\overline{D}\rightarrow\mathbb{R}$, and $\mathcal{Z}:=\int e^{-S(X,g)}dX$ is a renormalization constant, called the \textit{partition function}.
This random field being defined, we can consider a random measure on the surface by taking $d\mu=e^{\gamma\phi}d\lambda_{g}$. 

While the theory in the physical perspective is rather well understood, the mathematical study of these notions is relatively recent. Motivated by this physical background, David, Kupiainen, Rhodes and Vargas in their pioneering work \cite{DKRV16} launched a program whose goal was to provide a rigorous mathematical construction of the objects involved in the physics literature, program based on probability theory. To do so and in a subsequent series of work the authors along with Guillarmou and Huang rigorously constructed Liouville quantum field theory on Riemann surfaces with fixed topology---on spheres in \cite{DKRV16}, disks in \cite{HRV16}, tori in \cite{DRV16} and on higher genus surfaces in \cite{GRV16}---thanks to the introduction of a rigorous probabilistic framework, featuring two fundamental objects of probability theory: the \textit{Gaussian Free Field} and the \textit{Gaussian Multiplicative Chaos}.  
Under this framework, the field $\phi$ whose law is given by the Liouville action functional can be expressed in terms of the \textit{Gaussian Free Field} (GFF in the sequel), a random distribution on $D$, and that is in some sense an extension of the Brownian Motion where the variable of time is no longer one-dimensional but now lives in a two-dimensional space. The random measure then defined (formally) by the expression $\mu(dz)=e^{\gamma \phi(z)}d\lambda_g(z)$ is called a \textit{Gaussian Multiplicative Chaos} (GMC in the sequel). However the latter expression is purely heuristic, since the field $\phi$ is not defined pointwise. In order to provide a rigorous meaning to this formal writing, one uses a procedure of approximation of this field by smooth functions and then takes the limit of the corresponding measures (see \cite{DS10} for instance).
In order to determine the exact form of the field, the first step in the series of articles listed above consists in defining in rigorous terms the partition function studied in the physics literature, that does exist provided that one works under the so-called \textit{Seiberg bounds}. It is then possible to give a meaning to the random field defined by normalization via the partition function, which in turn gives rise to a random measure $\mu(dz)=e^{\gamma \phi(z)}d\lambda_g(z)$ on the surface. Heuristically this random measure is the volume form associated to the (formal) Riemannian metric $e^{\gamma\phi}g$, and is rigorously defined thanks to the procedure described above. This approach allowed the authors to recover fundamental properties predicted in the physics literature, such as the conformal Ward and BPZ identities in \cite{KRV15} or the DOZZ formula in \cite{KRV_18}.

On another perspective, Duplantier, Miller and Sheffield in their fundamental work \cite{DMS14} also provided definitions for similar objects, as suggested by Sheffield in \cite{S16}. In the latter, Sheffield defined according to a limiting procedure what he called the \textit{unit area quantum sphere} and the \textit{unit boundary length quantum disk} along with other natural random quantum surfaces such as \textit{quantum wedges} and \textit{quantum cones} and conjectured that these objects should be related to a scaling limit of uniform quadrangulations. In \cite{DMS14}, the authors provide a more explicit construction of these objects in terms of Bessel processes, and study their relationship with three key objects in the theory of random geometry: the Gaussian Free Field, the \textit{Schramm-Loewner Evolutions} and \textit{Continuum Random Trees}. These links allow them to derive many properties, among them the description of the scaling limit (in some sense) of a certain class of random planar maps in terms of CLE-decorated Liouville Quantum Gravity. 

These two perspectives differ on many points: first of all, the objects they consider do not actually live in the same space. In the first three articles, the authors explicitly constructed a random measure on a determined surface, after having picked three points which represent the singularities of the measure. Conversely, the quantum disk and sphere are actually defined in terms of \textit{quantum surfaces} in \cite{DMS14}, which are equivalence classes (modulo conformal maps) of surfaces endowed with a random measure. More precisely, two pairs $(D,h)$ (with $D$ a Riemann surface and $h$ a distribution on $D$) and $(\tilde{D},\tilde{h})$ are said to be equivalent when there exists a conformal mapping $\psi:\tilde{D}\rightarrow D$ such that $\tilde{h}=h\circ\psi+Q\log |\psi'|$ where $Q=\frac{\gamma}{2}+\frac{2}{\gamma}$. This defines an equivalence relation on the set of pairs $(D,h)$, and by doing so, the Liouville Quantum Gravity measures on $D$ defined by the distribution $h$ (that is the pair of random measures $\mu_h=e^{\gamma h}d\lambda$ and $\nu_h=e^{\frac{\gamma}{2} h}d\lambda_{\partial}$) do not actually depend on the representative of the equivalence class, since one has the property of change of variable as stated in \cite[Proposition 2.1]{DS10}: if we define a field $\tilde{h}$ on $\tilde{D}$ by setting
\begin{equation}\label{CCD2}
\tilde{h}:=h\circ\psi+Q\log |\psi'|
\end{equation}
then the pair of random measures $(\mu_{\tilde{h}}, \nu_{\tilde{h}})$ on $\tilde{D}$ (defined in the same as $\mu_h$ and $\nu_h$) and the pushforward under $\psi^{-1}$ of the measures $(\mu_h,\nu_h)$ on $D$ are almost surely equal. Showing that two quantum surfaces are equivalent is in general not obvious, and in the present case the two perspectives of defining quantum surfaces do not actually rely on similar procedures.
On the one hand, the first approach allowed the authors to provide an explicit expression for the law of the measures, but in order for the construction to make sense one needs to choose deterministically at least three points. On the other hand, it may be more convenient to rely on a limiting procedure to construct the object via the second approach, but for this construction to make sense is required to choose in a deterministic way only two points, other points (which facilitate the formulation of the equivalence) being picked at random.

A third definition for these objects could be to consider the quantum surface as the scaling limit of natural discrete random planar maps with the topology of this surface, an approach followed first by Le Gall and Miermont in \cite{LeG13} and \cite{Mie13} with the definition of the so-called \textit{Brownian map}, and then by Bettinelli and Miermont in \cite{BM18} with the \textit{Brownian disk}. In these articles, the authors defined the Brownian surface as a metric space, without consideration for the conformal structure, while the two perspectives we have studied so far construct a conformal structure on the surface for which it was unclear that the natural metric it comes with was well-defined. However in the series of article \cite{MS15a}, \cite{MS16a} and \cite{MS16b}, Miller and Sheffield constructed a metric on their quantum surfaces (the QLE-metric) and showed that their definitions coincide with the one given for the Brownian surfaces in the special case where $\gamma=\sqrt{\frac{8}{3}}$. Extending the definition of the metric to the whole range of $\gamma\in(0,2)$ has been achieved very recently (see \cite{DDDF19} and \cite{GM19}).

In the perspective of unifying different approaches, a recent article \cite{AHS16} by Aru, Huang and Sun showed that the two definitions given for the unit area quantum sphere define the same quantum surface. However a similar statement has not been proved yet in the case of the disk. This is the main result of this article.

Before moving on to the statement of the result, let us first give one application of this result in the realm the probability theory. In \cite{AG19}, the authors provide a computation of the conditional law of the area of the quantum disk when conditioned to have unit boundary length, which is expressed in terms of the so-called mating-of-trees variance constant and uses the mating-of-trees framework. Since the law of the area of the unit boundary length quantum disk computed following the approach by Huang, Rhodes and Vargas should actually be the same, the value of this constant may be recovered by using the Huang-Rhodes-Vargas approach. The computation of this constant would be of significant importance in the study of the scaling limit of some models of Random Planar Maps.

\paragraph{Acknowledgements} I am very grateful to J. Miller for having suggested this problem and for discussions and relectures. I am also thankful to E. Gwynne and G. Remy for having provided additional motivation for the problem. I would also like to thank the Statistical Laboratory of the Centre for Mathematical Sciences in Cambridge for the support and hospitality provided while this problem was being investigated, as well as the organisers of the program \textit{RGM Follow Up} that took place in the Isaac Newton Institute for Mathematical Sciences, Cambridge, during which I have been able to improve my understanding of the different perspectives. Eventually I would like to thank the anonymous referees for their numerous comments and suggestions that helped me to clarify the present paper.
\subsection{Statement of the equivalence and strategy of proof}

To give a precise statement of our main result, it is necessary to define precisely the two objects we will focus on in the sequel. Hence we fix a constant parameter $\gamma\in(0,2)$ throughout the rest of this subsection, and work in the unit disk $\mathbb{D}$.

In the first construction, one starts by choosing three distinct points $z_1,z_2,z_3$ on the boundary of the unit disk $\partial \mathbb{D}$, and construct a random field $h_L$, which is a GFF---with free boundary conditions and zero mean on the boundary of the disk $\partial\mathbb{D}$---to which have been added the corresponding \textit{log-singularities}
\begin{equation}
    - \sum_{i=1}^3 \gamma \log|z-z_i|.
\end{equation}
Using this field, we define a pair of random measures on $\mathbb{D}$ (called the \textit{bulk measure}) and $\partial \mathbb{D}$ (called the \textit{boundary measure}) thanks to the theory of GMC by taking (formally) $\mu_{h_L}(dz):=e^{\gamma h_L}\lambda(dz)$ and $\nu_{h_L}(dz):=e^{\frac{\gamma}{2} h_L}\lambda_{\partial}(dz)$, where $\lambda$ and $\lambda_{\partial}$ are the Lebesgue measure on $\mathbb{D}$ and $\partial \mathbb{D}$. 
Once these random measures well-defined, the law of the unit boundary length quantum disk may be set as the law of the pair of measures thus defined under some weighted probability measure. The precise meaning of this construction will be detailed in Subsection \ref{def_HRV}.
We refer to the pair of random measures thus constructed as the \textit{unit boundary length quantum disk with three log-singularities}, which we will denote $(\mu_{\mathrm{HRV}}^{\mathrm{UBL}},\nu_{\mathrm{HRV}}^{\mathrm{UBL}})$. 

Let us now turn to the second perspective. Recall the definition of the notion of quantum surface as a class equivalence of pairs $(D,h)$ modulo conformal mappings, with the rule of change of variable given by Equation \eqref{CCD2}.
This equivalence relation may be extended to include the notion of \textit{marked points}. For any $(x_1...x_k)\in D^k$ and $(y_1...y_k)\in \tilde{D}^k$ we may assume that, in addition to Equation \eqref{CCD2}, $\psi(y_i)=x_i$ for any $1\leq i\leq k$: a class equivalence of such $(k+2)$-tuples is called a \textit{quantum surface with $k$ marked points}. Here, the unit boundary length quantum disk is a random quantum surface with three marked points, which has the law of $(\mathbb{D}, h, -1, 1, z_3)$ where $h$ is a random distribution on $\mathbb{D}$ constructed thanks to an encoding with Bessel processes and $z_3$ is sampled on the boundary of $\mathbb{D}$ from the measure $\nu_h$, where we have defined $(\mu_h,\nu_{h})=(e^{\gamma h(z)}\lambda(dz),e^{\frac{\gamma}{2} h(z)}\lambda_{\partial}(dz))$. We will refer to the law of this quantum surface as the \textit{unit boundary length quantum disk with three marked points}. Again, the precise definition of this object will be explained in more details in Subsection \ref{def_DMS}.

In order to state an equivalence between these two objects, we can notice that for any three distinct points $(z_1,z_2,z_3)$ on the boundary of the disk, there is a unique representative of the unit boundary length quantum disk whose marked points are $(z_1,z_2,z_3)$, which we call an embedding of the unit boundary length quantum disk with marked points $(z_1,z_2,z_3)$. We denote the law of the measures obtained in this embedding $(\mu_{DMS}^{UBL},\nu_{DMS}^{UBL})$. We are now ready to state our main result:
\begin{theorem}[Equivalence of the perspectives]\label{Equ}\hspace{0cm}\\
Let $\mathbb{D}$ be the unit disk and $(z_1,z_2,z_3)$ be distinct points on its boundary.
Let $(\mu_{HRV}^{UBL},\nu_{HRV}^{UBL})$ be the unit boundary length quantum disk with three log-singularities.
Likewise assume that $(\mu_{DMS}^{UBL},\nu_{DMS}^{UBL})$ in an instance of the unit boundary length quantum disk with three marked points embedded into $\mathbb{D}$ so that the marked points are $(z_1,z_2,z_3)$.

Then $(\mu_{HRV}^{UBL},\nu_{HRV}^{UBL})$ and $(\mu_{DMS}^{UBL},\nu_{DMS}^{UBL})$ have same law.
\end{theorem}

In \cite{BSS14}, Berestycki, Sheffield and Sun proved that the measure constructed from a field which is locally mutually absolutely continuous with respect to a GFF actually determines the field from which it has been defined. Briefly after the statement of the main result \cite[Theorem 1.1]{BSS14} of the article, the authors claim that it can be applied in the two contexts we have exposed. This  allows us to give a similar statement in terms of the underlying fields, and therefore in terms of quantum surfaces.
\begin{corollary}[Equivalence of the perspectives, alternative formulation]\label{statement_quantum}
\hspace{0cm}\\
Let $\mathbb{D}$ be the unit disk and $(z_1,z_2,z_3)$ be distinct points on its boundary. Let $h_0$ be a Gaussian Free Field on $\mathbb{D}$ with free boundary conditions and mean zero on $\partial\mathbb{D}$, \[
h_L:=h_0 - \sum_{i=1}^3 \gamma \log|z-z_i|
\]
and let $h^*$ have the law of the field $h_L-\frac{2}{\gamma}\log \nu_{h_L}(\partial\mathbb{D})$ under the weighted probability measure $\propto \nu_{h_L}(\partial\mathbb{D})^{\frac{2Q-3\gamma}{\gamma}} d\mathbb{P}$. 
Then the quantum surface $(\mathbb{D},h^*, z_1,z_2,z_3)$ has the law of the unit boundary length quantum disk with three marked points.
\end{corollary}

The two constructions that we have given are rather different, and in most cases it is not obvious that two laws on fields $h$ induce equivalent quantum surfaces. However, let us give some intuition of why such a result may be true: first of all, the choice of three marked points on its boundary fixes a conformal structure on the disk, but still all the disks with three marked points on the boundary are conformally equivalent, which is no longer true if we choose four or more marked points on its boundary. Secondly, these two objects are both conjecturally related to scaling limits of some Random Planar Maps models for the whole range of $\gamma\in(0,2)$: for instance to random quadrangulations with the topology of the disk for the special value of $\gamma=\sqrt{\frac{8}{3}}$, as stated in \cite[Section 5]{HRV16} and \cite[Section 6]{S16}.

The structure of the article can be described as follows. To start with, Section \ref{general_defi} is dedicated to providing the analytical and probabilistic background necessary in the sequel. We then study in Section \ref{two_persp} the two different perspectives and highlight a limiting procedure from \cite{DMS14} leading to the unit boundary length quantum disk with \textit{three marked points}. We then show that a slight modification of this procedure gives in the limit the unit boundary length quantum disk with \textit{three log-singularities}, and that the two limiting laws are actually the same by noticing that the perturbation becomes negligible in the limit. This is the content of Section \ref{equal}.

\section{General setting and definitions}\label{general_defi}
Throughout this document we will consider complex domains whose boundary (when non-empty) consists of finitely many lines, circles or semi-circles; we denote without loss of generality by $D$ such a domain and by $\partial D$ its boundary. We may also introduce the notations $\lambda$ and $\lambda_{\partial}$ that stand for the Lebesgue measures on $D$ and $\partial D$. The purpose of this section is to expose briefly the objects we will work with in the sequel.
\subsection{Analytic background}

\subsubsection{Sobolev spaces}
In the rest of the article, we will work on the general functional spaces defined below.
We start by considering the case where $D$ is different from the whole plane, and define $\Hr_\mathrm{s}(D)$ to be the set of smooth functions $f:D\rightarrow \mathbb{R}$ with compact support included in $D$ (we refer to these as Dirichlet or zero boundary conditions). Likewise $\Hr_{\partial}(D)$ is the set of smooth functions with mean zero on the boundary of $D$ (also known as Neumann or free boundary conditions).\\
We endow these spaces with the \textit{Dirichlet inner product} $(\cdot{},\cdot{})_{\nabla}$ defined by \begin{equation}
(f,g)_{\nabla}:=\frac{1}{2\pi}\int_D (\nabla f(z)\cdot\nabla g(z))\lambda(dz)
\end{equation}
which has the fundamental property to be invariant under conformal mapping in dimension 2.\\ 
Then we denote by $\Hr(D)$ and $\Hr_\mathrm{N}(D)$ the Hilbert space completion of $\Hr_\mathrm{s}(D)$ and $\Hr_{\partial}(D)$ when endowed with the $L^2$ norm associated to the Dirichlet inner product.
We can also take into account the fact that the domain $D$ has a boundary as follows. Let $L$ be a linear part of $\partial D$ and define $\Hr_\mathrm{m}(D)$ as the set of smooth functions with compact support included in $D\cup L$. The Sobolev space $\Hr_\mathrm{M}(D)$ with mixed boundary conditions is then defined by taking the Hilbert space completion of $\Hr_\mathrm{m}(D)$ with respect to the Dirichlet inner product. In terms of GFF this means that the field has free boundary conditions on $L$ and zero boundary conditions on $\partial D\setminus L$.
In the case of the whole plane, we follow the definition from \cite{MS13}, by working in the completion of the set of smooth functions with compact support and zero mean on $\mathbb{C}$ endowed with the Dirichlet inner product.

In the sequel we will denote by $\mathbb{P}_{D}$ the set of probability measures of the form $\rho(z)\lambda(dz)$ on $D$, that is to say the set of smooth functions $D\to\mathbb{R}^+$ with $\int_D \rho(z)\lambda(dz)=1$ which we refer to as the \lq\lq background measures". For any such $\rho$ in $\mathbb{P}_D$ and $f$ in one of the Hilbert spaces we set \begin{equation}
m_{\rho}(f):=\int_D f(z)\rho(z)\lambda(dz).
\end{equation}
When $D$ is bounded and $\rho$ is the uniform probability measure on $D$, we will omit the dependence in $\rho$ in order not to overload the notations.
We may proceed in the same way to define $\mathbb P_{D, \partial}$ as the set of probability measures on the closure $\overline D$ of $D$ that are supported on a line or a (semi-)circle $L$ and similarly denote 
\begin{equation}
m_{\rho,\partial}(f):=\int_{L} f(z)\rho(z)\lambda_{\partial}(dz).
\end{equation}
We will often forget about the $\partial$ in the notation of $m_{\rho,\partial}$; more generally if $S$ is some part of $\partial D$ we may denote simply $m_{S}:=m_{\rho,\delta}$ where $\rho$ is the uniform probability measure on $S$.  

\subsubsection{Orthogonal decompositions of the Sobolev spaces}
The GFF is known to enjoy a so-called \textit{Markov property}, in the sense that to an orthogonal decomposition of the Sobolev spaces described above is associated an independent decomposition for the associated GFF. In the sequel we will make use of the following orthogonal decompositions of the Sobolev spaces. We introduce the notation $B(x,r)$ for the Euclidean ball of radius $r$ and centered at $x\in\mathbb C$. Our first decomposition applies to $D=\mathbb{H}$ the upper half-plane endowed with $\Hr_N(\mathbb{H})$ the Sobolev space with free boundary conditions on $\partial\mathbb{H}$.
\begin{prop}[Radial-angular decomposition]\label{OD1}
\hspace{0cm}\\
Let $\Hr_\mathrm{ang}(\mathbb{H})$ (resp. $\Hr_\mathrm{rad}(\mathbb{H})$) be the Hilbert space completion of the set of the $f\in H_s(\mathbb H)$ with mean zero on every semi-circle $\partial B(0,r)\cap \mathbb{H}$ (resp. constant on every semi-circle $\partial B(0,r)\cap \mathbb{H}$). 

Then $\Hr_N(\mathbb H)=\Hr_\mathrm{ang}(\mathbb H)\oplus \Hr_\mathrm{rad}(\mathbb H)$.
\end{prop}
\begin{proof}
On the one hand, for any smooth $g$ that is contant on every semi-circle we know that $\nabla g(z)$ is orthogonal to the semi-circle of radius $|z|$, while and its modulus only depends on $|z|$; therefore its mean value on the semi-circle is vertical. On the other hand, for any smooth $f$ that has zero mean on every semi-circle the mean value of $\nabla f(z)$ on any semi-circle is horizontal. Since taking the Dirichlet inner product of $f$ and $g$ consists of summing the scalar product of these mean values, we see that $(f,g)_{\nabla}=0$. Now, $f\in \Hr_{\partial}(\mathbb H)$ can be written as the (orthogonal) sum $f=\left(f-g\right) +g$, where $g(r)=m_{\partial B(0,r)\cap \mathbb{H}}(f)$. Since $\Hr_N(\mathbb H)$ is defined as the Hilbert space completion of $\Hr_{\partial}(\mathbb H)$ with respect to the Dirichlet norm this concludes the proof.
\end{proof}
We now turn to the case of half-disks, which will be central in the proof of our main result. In the following statement we consider any positive $R$ and $D$ the semi-disk $R\mathbb{D}\cap\mathbb{H}$. On $D$ we let $\Hr_M(D)$ be the Sobolev space with free (resp. zero) boundary conditions on $[-R,R]$ (resp. $\partial B(0,R)\cap\mathbb{H}$). 
\begin{prop}[Circle-average decomposition]\label{OD2}
\hspace{0cm}\\
For any $0<r<R$, let $\Hr_\mathrm{c}(D)$ be the Hilbert space completion of the set of the $f\in \Hr_m(D)$ with zero mean on the semi-circle $\partial B(0,r)\cap \mathbb{H}$ and
$\Hr_{\xi}(D):=\left\lbrace t\xi_r, t\in \mathbb{R} \right\rbrace$, where 
\[
\xi_r(z)=-2\log(\max(r,|z|))+2\log R.
\]
Then $\Hr_M(D)=\Hr_\mathrm{c}(D)\oplus \Hr_{\xi}(D)$.
\end{prop}
\begin{proof}
Denote by $\rho_r=\frac{1}{|\partial B(0,r)\cap \mathbb{H}|}\mathsf{1}_{\partial B(0,r)\cap \mathbb{H}}$ the uniform (probability) measure on the semi-circle $\partial B(0,r)\cap\mathbb{H}$.
We know from \cite[Subsection 6.1]{DS10} that for any element $\phi$ of $H_m(D)$ we have $(\phi,\xi_r)_{\nabla}=(\phi,\rho_r)$, and $(\xi_r,\xi_r)_{\nabla}=2\log R$.\\
Therefore for any smooth $f$ with zero mean on $\partial B(0,r)\cap \mathbb{H}$ we have that $(f,\xi_r)_{\nabla}=(f,\rho_r)=0$, so any $f\in H_m(D)$ can be written as the (orthogonal) sum $f=\left(f-\lambda\xi_r\right) + \lambda\xi_r$, where $\lambda=\displaystyle\frac{m_{\rho_r,\partial}(f)}{2\log R}$. Taking the Hilbert space completion of this decomposition yields the result.
\end{proof}

\subsubsection{Green's kernel} On the domain $D$, consider one of the previous functional spaces $\Hr_a(D)$ (either $\Hr_\mathrm{s}(D)$, $\Hr_{\partial}(D)$ or $\Hr_\mathrm{m}(D)$) and $\Hr_A(D)$ to be its Hilbert space completion. We define the \textit{Green's kernel} $G_D$ associated to the functional space $\Hr_A(D)$ to be the unique symmetric kernel with the properties that: \begin{itemize}
    \item For any $f$ in $\Hr_a(D)$ and $x$ in $D$:
\begin{equation}\label{IPP}
\int_D (-\Delta) f(y)G(x,y)\lambda(dy)=2\pi \left(f(x)-m_{\partial D}(f)\right)-\int_{\partial D} G(x,y) \partial_{n} f(y)\lambda_{\partial}(dy) .
\end{equation} 
\item For any $x$ in $\overline{D}$, the map $z\mapsto G_D(x,z)$ satisfies the same property that the elements of $\Hr_A(D)$ (\textit{e.g.} zero boundary condition for the Dirichlet problem and $m_{\partial D}(G_D(x,\cdot))=0$ for Neumann boundary conditions).
\end{itemize}

Such a kernel indeed exists and is characterized (in the case of free boundary conditions) as the unique symmetric solution of the following Neumann problem: \\
For any $y\in D$, $x\mapsto G(x,y)$ has the properties of:
\begin{itemize}
\item harmonicity on $D\setminus\lbrace y\rbrace$,
\item harmonicity on $D$ of $x\mapsto G(x,y)+ \log |x-y|$,
\item $\partial_n G_D(x,y)=-\frac{2\pi}{|\partial D|}$ for $x\in\partial D$ ($0$ if the boundary is unbounded), where $\partial_n$ is the normal derivative,
\item mean zero on $\partial D$.
\end{itemize}
Note that since we work only in specific domains there are no boundary issues. Interestingly, this kernel can be made explicit in some cases:
\begin{align}
G_{\mathbb{D}}(x,y)&=-\log\left( |x-y||1-xy^*|\right) \quad\text{for the unit disk }\overline{\mathbb{D}}\\ 
G_{\mathbb{H}}(x,y)&=-\log \left(|x-y||x-y^*|\right)\quad\text{for the upper half-plane }\overline{\mathbb{H}} 
\end{align}
both with free boundary conditions.

In the same spirit, we introduce a larger set of Green's kernel by requiring it to have mean zero on $D$ or $\partial D$ under a different metric. To do so, we define for any $\rho\in\mathbb{P}_D$ 
\begin{equation}\label{def:green_rho}
G_D^{\rho}(x,y):=G_D(x,y)-m_{\rho}(G_D(x,\cdot))-m_{\rho}(G_D(\cdot,y))+\theta_{\rho}
\end{equation} 
with $\theta_{\rho}:=\iint_{D\times D} \rho(x)G_D(x,y)\rho(y)\lambda(dx)\lambda(dy)$ chosen so that $m_{\rho}(G_D^{\rho}(x,\cdot))=0$ for any $x$ in $D$. We may proceed in the same way for the boundary case and define 
\begin{equation}
G_{D,\partial}^{\rho}(x,y):=G_D(x,y)-m_{\rho,\partial}(G_D(x,\cdot))-m_{\rho,\partial}(G_D(\cdot,y))+\theta_{\rho,\delta}
\end{equation} 
with $\theta_{\rho,\delta}:=\iint_{L\times L} \rho(x)G_D(x,y)\rho(y)\lambda_\partial(dx)\lambda_\partial(dy)$.

\subsection{Probabilistic background: Gaussian Free Field and Gaussian Multiplicative Chaos}

\subsubsection{Gaussian Free Field}
Roughly speaking, the GFF is a $d$ time-dimensional analog of the Brownian Motion, which can be seen both as a random distribution over a domain $D$ and a Gaussian Hilbert space. 
Following the approach of Janson in \cite{Jan97}, we may define the \textit{Gaussian Free Field with Dirichlet, Neumann or mixed boundary conditions} as the Gaussian Hilbert space whose random variables are the $ (h,f)_{\nabla}$ for $f$ in $\Hr$, where $\Hr$ is one of the Hilbert spaces $\Hr(D)$, $\Hr_N(D)$ or $\Hr_M(D)$, with the property that these random variables are Gaussian with mean zero and covariance function given by cov$((h,f)_{\nabla},(h,g)_{\nabla})=(f,g)_{\nabla}$. It is important to notice that at this stage, in the case of free boundary conditions as well as in the case of the whole plane, the GFF is defined modulo an additive constant. Usually in order to set the value of this constant one further assumes that the field has zero mean on $\partial D$ in the free boundary case. Standard computations using Equation \eqref{IPP} then show that the GFF may be thought of as a Gaussian field with covariance kernel is given by $G_{D}$.

One may proceed in the same way and define more generally the GFF $h_{\rho}$ associated to some $\rho$ in $\mathbb{P}_D$ (resp. $\mathbb{P}_{D,\partial}$) as a Gaussian field whose covariance kernel is given by $G_D^\rho$ (resp. $G_{D,\partial}^\rho$); existence and properties of such fields are detailed for instance in \cite{S07} or \cite{DS10}. If we do so then $h_{\rho}$ can be thought of as a GFF on $D$ such that $(h,\rho)=0$. 

From invariance under conformal mapping of the Dirichlet inner product can be raised the property of invariance under conformal mapping of the GFF; similarly it is standard (see \cite{S07} for instance) that the field enjoys a sort of Markov property that allows it to be decomposed into independent Gaussian components, and that this field is highly non regular and lives in the Sobolev space with negative index $\Hr^{-1}$, that is the dual of the Sobolev space $\Hr$. Besides, a crucial property of the GFF is that this random distribution can give rise to a random measure on $D$, usually referred to as the Liouville Quantum Gravity measure. This random metric can formally be written under the form $e^{\gamma h(z)}\lambda(dz)$; however, since $h$ is a distribution and cannot be defined pointwise, we will use an approximation process to make this definition precise.

\subsubsection{Regularization of the GFF: circle averages}
Let $h_0$ be a GFF on $D$ (with one of the three boundary conditions).\\
For $z$ in $D$ or in a linear part $L$ of $\partial D$, we would like to define for $\varepsilon>0$ small enough $h_0^{\varepsilon}(z)$ its mean value on the circle/semi-circle $\partial B(z,\varepsilon)\cap D$. This random variable is actually well-defined, since this mean value can be written under the form $(h,\zeta_z^{\varepsilon})_{\nabla}$ for some $\zeta_z^{\varepsilon}$ in $H$ (see for instance \cite[Section 3]{DS10}). For the sake of completeness, we provide here the explicit construction in the case of a semi-circle in the the upper half-plane, since we will use it in the sequel.

Let $h_0$ be a GFF with free (resp. zero) boundary conditions on $[-R,R]$ (resp. $\partial B(0,R)\cap\mathbb{H}$). For $r<R$, let $\rho_{r}$ and $\xi_r$ be as in Proposition \ref{OD2}. Then by integration by parts we have $(h_0,\rho_r)=(h_0,\xi_r)_{\nabla}$, so we can define the semi-circle average of $h$ over $\partial B(0,r)\cap\mathbb{H}$. This random variable is therefore Gaussian with mean zero and variance $(\xi_r,\xi_r)_{\nabla}=2\log \frac{R}{r}$.

\subsubsection{Gaussian Multiplicative Chaos and Liouville Quantum Gravity}
We are now ready to define for $\gamma\in(0,2)$ the so-called \textit{Liouville Quantum Gravity} measures on $D$ and (a linear part of) $\partial D$. In the following statement we assume $L$ to be a linear part of $\partial D$ and let $h_0$ be a GFF with mixed boundary conditions.
\begin{theorem}\label{LQG}
Let $\gamma\in(0,2)$ and $h_0^{\varepsilon}(z)$ be the $\varepsilon$-circle average around $z$ if $B(z,\varepsilon)\subset D$, and the $\varepsilon$-semi-circle average around $z$ if $z$ is in $L$---of the field $h_0$.
Then the sequence of random measures on $D\times L$ defined by $\left(\varepsilon^{\gamma^2/2}e^{\gamma h_0^{\varepsilon}(z)}\lambda(dz), \varepsilon^{\gamma^2/4}e^{\frac{\gamma}{2} h_0^{\varepsilon}(z)}\lambda_{\partial}(dz)\right)$ converges almost surely in the sense of weak convergence of measures as $\varepsilon$ goes to zero. We denote their limit by $\mu_{h_0}(dz)=e^{\gamma h_0}\lambda(dz)$ and $\nu_{h_0}(dz)=e^{\frac{\gamma}{2} h_0}\lambda_{\partial}(dz)$. 
\end{theorem}
If $D$ is a domain with non-linear boundary but can be mapped conformally to a domain $\tilde{D}$ with linear boundary, we may define its boundary measure to be the pushforward of the boundary measure of $\tilde{D}$; it is a standard result (see the proof of \cite[Theorem 6.1]{DS10} for instance) that this definition is actually consistent with the change of variable formula \eqref{CCD2}.

The proof of this classical result can be found for instance in \cite{DS10}, and can also be extended to the case where the field is shifted by a deterministic constant. We will consider in the sequel Liouville Quantum Gravity measures that are defined from fields of the form
\begin{equation}\label{GFF_L}
h_{L(\rho)}=h_{\rho}+ Q m_{\rho}(G_D(z,\cdot))+H_{\rho}(z)
\end{equation} where 
\begin{equation}
 H_{\rho}(z)=\sum_{i=1}^n \alpha_i G_D^{\rho}(z,z_i) + \sum_{j=1}^r \frac{\beta_j}{2} G_D^{\rho}(z,s_j)
\end{equation}
and with pairs $(\bm\alpha,\bm z):=(\alpha_1,z_1),...,(\alpha_n,z_n)$ in $\mathbb{R}\times D$ and $(\bm \beta,\bm s):=(\beta_1,s_1),...,(\beta_r,s_r)$ in $\mathbb{R}\times \partial D$. Here the background measure is $\rho\in\mathbb{P}_{D}$, but the same remains true for $\rho\in\mathbb{P}_{D,\partial}$. For $\gamma\in(0,2)$, $\mu_{h_{L(\rho)}}=e^{\gamma h_{L(\rho)}}d\lambda$ and $\nu_{h_L(\rho)}=e^{\frac{\gamma}{2} h_{L(\rho)}}d\lambda_{\partial}$ will be referred to as \textit{the Gaussian Multiplicative Chaos}(GMC in the sequel) associated to $(\bm\alpha,\bm z)$ and $(\bm\beta,\bm s)$ in the rest of this document. These random measures are defined using a similar limiting procedure as in Theorem \ref{LQG}; see for instance \cite{RV16} for a justification of the construction of such an object, first introduced by Kahane in \cite{K85}, thanks to the theory of GMC.

\section{A review of the two perspectives}\label{two_persp}

In this third section, we review the two definitions of the unit boundary length quantum disk provided in \cite{HRV16} and \cite{DMS14} and shed light on some of their properties that will be useful for our purpose. In particular we will highlight a limiting procedure leading to the unit boundary length quantum disk with three marked points which is the starting point to showing the equivalence between these two perspectives.

\subsection{The unit boundary length quantum disk from the path integral}\label{def_HRV}
According to what has been said in the first section, the definition provided by Huang, Rhodes and Vargas in \cite{HRV16} comes from an interpretation of the path integral approach thanks to the introduction of probabilistic objects. 

\subsubsection{The partition function of the path integral from the probabilistic viewpoint} The starting point in this first approach is to provide a rigorous meaning to Liouville action functional that was introduced in Equations \eqref{Liouville_action} and \eqref{law_field}:
\begin{equation*}
    \expect{F(\phi)}=\frac{1}{\mathcal{Z}}\int F(X)e^{-S(X,g)}dX
\end{equation*}
with
\begin{equation*}
    S(X,g)=\frac{1}{4\pi}\int_{\mathbb{D}} (|\nabla_g X|^2 + Q R_g X + 4\pi\mu e^{\gamma X}) d\lambda_{g} + \int_{\partial \mathbb{D}} (Q K_g X + 2\pi\mu_{\partial}e^{\frac{\gamma}{2} X})d\lambda_{\partial g}.
\end{equation*}
The geometric term $e^{-\frac{1}{4\pi}\int_{\mathbb{D}} |\nabla_g X|^2d\lambda_{g}}$ in the action functional corresponds to the partition function of the two-dimensional GFF: it is therefore natural to interpret the measure $e^{-\frac{1}{4\pi}\int_{\mathbb{D}} |\nabla_g X|^2d\lambda_{g}}dX$ as the probability measure with respect to some GFF. However the partition function of the theory is not well-defined and requires the insertion of \textit{logarithmic singularities} to the field (which amounts to adding \textit{conical singularities} to the underlying quantum surface). This is done as follows.

Let $(\mu,\mu_{\partial})\in\mathbb{R^+}\times\mathbb{R^+}\setminus\lbrace(0,0)\rbrace$ be a pair of cosmological constants, and similarly $(\bm\alpha,\bm z)$ be in $\mathbb{R}\times \mathbb{D}$ and $(\bm\beta,\bm s)$ be in $\mathbb{R}\times \partial \mathbb{D}$. When considering such pairs we set \begin{equation}\label{def:s}
    s:=\sum_{i=1}^n \alpha_i  + \sum_{j=1}^r \frac{\beta_j}{2} - Q.
\end{equation}
By considering functionals to which we have added \lq\lq vertex operators" $V_{\alpha}(x):=e^{\alpha\phi(x)}$, we introduce the partition functions for the Liouville field in the unit disk, which take the form: \begin{equation}
    \hspace{-0.5cm}\Pi_{\gamma,\mu,\mu_{\partial}}^{(\bm z,\bm{\alpha}),(\bm s,\bm{\beta})}(F):= \int_{\mathbb{R}} e^{sc}\mathbb{E}\left[F(h_{L(\rho_0)}+c)\exp\left(-\mu e^{\gamma c}\mu_{h_{L(\rho_0)}}(\mathbb{D})-\mu_{\partial} e^{\frac{\gamma}2 c}\nu_{h_{L(\rho_0)}}(\partial \mathbb{D})\right)\right]dc,
\end{equation}
where we have considered the field $h_L(\rho_0)$ to be given by Equation \eqref{GFF_L} in the special case where the background measure is $\rho_0$, the uniform probability measure on the boundary of the disk. Like in the introduction, the normalization constant is defined by 
\begin{equation}
\mathcal{Z}:=\Pi_{\gamma,\mu,\mu_{\partial}}^{(\bm z,\bm{\alpha}),(\bm s,\bm{\beta})}(1)=\int_{\mathbb{R}} e^{sc}\mathbb{E}\left[\exp\left(-\mu e^{\gamma c}\mu_{h_{L(\rho)}}(\mathbb{D})-\mu_{\partial} e^{\frac{\gamma}2 c}\nu_{h_{L(\rho)}}(\partial \mathbb{D})\right)\right]dc.
\end{equation}
It is shown by the authors that the above quantity is indeed well-defined provided that the \textit{Seiberg bounds} are satisfied:\begin{itemize}
    \item if $\mu>0$: 
\begin{equation}\label{Seiberg}
        s>0 \text{;\hspace{0.15cm}for any } 1\leq i\leq n, \text{ }\alpha_i<Q \text{;\hspace{0.15cm} for any } 1\leq j \leq r, \text{ }\beta_j<Q.
\end{equation}
\item if $\mu=0$:
\begin{equation}
        s>0 \text{;\hspace{0.15cm} for any } 1\leq j \leq r, \text{ }\beta_j<Q.
\end{equation}
\end{itemize}
These bounds correspond to the facts that the log-singularities have to be integrable ($\alpha,\beta<Q$) and that there is no divergence in the zero mode ($s>0$), that is when the variable $c$ in the partition function diverges to $-\infty$. Geometrically this last bound corresponds to the fact that for a disk to admit a Riemannian metric whose (both Gaussian and geodesic) curvature is constant and negative, one needs to add certain conical singularities in order to overcome the Gauss-Bonnet theorem.

\subsubsection{Law of the Liouville field and measures} Under these assumptions, the Liouville field is the random field whose law is given by the normalization of the previous expression: 
\begin{equation}
    \mathbb{E}_{\gamma,\mu,\mu_{\partial}}^{(\bm z,\bm{\alpha}),(\bm s,\bm{\beta})}\left[F(\phi)\right]:=\frac{\Pi_{\gamma,\mu,\mu_{\partial}}^{(\bm z,\bm{\alpha}),(\bm s,\bm{\beta})}(F)}{\Pi_{\gamma,\mu,\mu_{\partial}}^{(\bm z,\bm{\alpha}),(\bm s,\bm{\beta})}(1)}.
\end{equation} Following a process similar to the one described in the introduction, this field then gives rise to a pair of random measures $(M,M^{\partial})$ on $(\mathbb{D},\partial\mathbb{D})$ (corresponding to GMC measures associated to $\phi$), whose law is described by an equation of the form (see \cite[Subsection 3.6]{HRV16}):
\[\mathbb{E}_{\gamma,\mu,\mu_{\partial}}^{(\bm z,\bm{\alpha}),(\bm s,\bm{\beta})}\left[F(M,M^{\partial})\right]\propto\]
\begin{equation}
\int_{0}^{+\infty} y^{\frac{2s}{\gamma}-1}\mathbb{E}\left[F(y^2\frac{\mu_{h_L}}{\nu_{h_L}(\partial \mathbb{D})^{2}}, y\frac{\nu_{h_L}}{\nu_{h_L}(\partial \mathbb{D})})\exp(-\mu y^2\frac{\mu_{h_L}(\mathbb{D})}{\nu_{h_L}(\partial\mathbb{D})^{2}}-\mu_{\partial} y)\nu_{h_L}(\partial \mathbb{D})^{-\frac{2s}{\gamma}}\right]dy.
\end{equation}

In particular the following properties hold: 
\begin{itemize}
\item if $\mu=0$ (so $\mu_{\partial}>0$), then one can define the law of the \textit{unit boundary length quantum disk} with log-singularities $(\bm\alpha,\bm z)$ and $(\bm\beta,\bm s)$ as the law of the pair $(M,M^{\partial})$ conditioned on $M^\partial(\partial\mathbb{D})=1$. We denote it by $(\mu^{UBL}_{HRV},\nu^{UBL}_{HRV})$ (we omit the dependence in $(\bm\alpha,\bm z)$ and $(\bm\beta,\bm s)$ to keep the notations as light as possible). 
Conditioning on the value of the boundary length is somehow tantamount to fixing the value of the $y$-variables in the integral to be equal to $1$\footnote{To see this one can introduce the event that $M_\partial(\partial \mathbb D)\in(1-\varepsilon,1+\varepsilon)$ which corresponds to taking $y$ between $1-\varepsilon$ and $1+\varepsilon$; one can condition on this event and the let $\varepsilon$ go to zero to get the expression \eqref{UBL}.}; as a consequence the law of $(\mu^{UBL}_{HRV},\nu^{UBL}_{HRV})$ can be more explicitly described by \begin{equation}\label{UBL}
\mathbb{E}_{\gamma,\mu,\mu_{\partial}}^{(\bm z,\bm{\alpha}),(\bm s,\bm{\beta})}\left[F(\mu^{UBL}_{HRV},\nu^{UBL}_{HRV})\right]=
\frac{\mathbb{E}\left[F(\frac{\mu_{h_L}}{\nu_{h_L}(\partial \mathbb{D})^2}, \frac{\nu_{h_L}}{\nu_{h_L}(\partial \mathbb{D})})\nu_{h_L}(\partial\mathbb{D})^{-\frac{2s}{\gamma}}\right]}{\mathbb{E}\left[\nu_{h_L}(\partial \mathbb{D})^{-\frac{2s}{\gamma}}\right]}.
\end{equation}
\item if $\mu_{\partial}=0$ (so $\mu>0$), we can proceed in the same way to define the law of the \textit{unit area quantum disk} with log-singularities $(\bm\alpha,\bm z)$ and $(\bm\beta,\bm s)$: this simply corresponds to the law of the pair $(M,M^{\partial})$ when conditioned on $M(\mathbb{D})=1$. We denote it by $(\mu^{UA}_{HRV},\nu^{UA}_{HRV})$, which admits the alternative definition:
\begin{equation}\label{UA}
\mathbb{E}_{\gamma,\mu,\mu_{\partial}}^{(\bm z,\bm{\alpha}),(\bm s,\bm{\beta})}\left[F(\mu^{UA}_{HRV},\nu^{UA}_{HRV})\right]=
\frac{\mathbb{E}\left[F(\frac{\mu_{h_L}}{\mu_{h_L}(\mathbb{D})}, \frac{\nu_{h_L}}{\mu_{h_L}(\mathbb{D})^{1/2}})\mu_{h_L}(\mathbb{D})^{-\frac{s}{\gamma}}\right]}{\mathbb{E}\left[\mu_{h_L}(\mathbb{D})^{-\frac{s}{\gamma}}\right]}.
\end{equation}
\end{itemize} 
As explained in \cite{HRV16}, these measures can actually be defined under the following weaker assumptions:
\begin{equation}
 -s<\min(\frac{2}{\gamma},2(Q-\alpha_i),Q-\beta_j) \text{ and for any } 1\leq i\leq n \text{ and } 1\leq j \leq r, \text{ }\alpha_i<Q \text{ and } \beta_j<Q
\end{equation} for the unit area quantum disk and 
\begin{equation}
 -s<\min(\frac{2}{\gamma},Q-\beta_j) \text{ and for any } 1\leq j \leq r, \text{ } \beta_j<Q
\end{equation} for the unit boundary length quantum disk.
When we consider another domain $\tilde{D}$ which is conformally equivalent to $\mathbb{D}$ under $\psi$, we define the Liouville measures to be the pushforwards of the corresponding measures on $\mathbb{D}$ by $\psi$. It is worth noting that this definition is consistent when we consider M\"obius transforms of the disk, as stated in \cite[Theorem 3.5]{HRV16}, as well with the rule of change of domain \eqref{CCD2} (see \cite[Proposition 3.7]{HRV16}).

There is one specific case which deserves special attention, and which is the one that we will consider in the sequel: it is given by taking three $\gamma$-singularities located on the boundary of the disk, \textit{e.g.} we consider three insertion points $(\gamma,1)$, $(\gamma,i)$ and $(\gamma,-1)$.

It is also interesting to note that thanks to similar constructions, it is possible to define the unit area quantum surface for the sphere \cite{DKRV16} and the torus \cite{DRV16}.

In the next subsection we provide some properties of these Liouville measures, which are direct consequences of the definitions given.

\subsubsection{First properties of the Liouville measure}
To begin with, we are interested in the change of coordinates associated to a conformal change of domain which would preserve the law of the pair of measures that define the law of the unit boundary length quantum disk. We provide the following proposition as a reminder of Equation \eqref{CCD2}. This statement follows from \cite[Proposition 3.7]{HRV16}.
\begin{prop}[Conformal change of domain]\label{CCD}
\hspace{0cm}\\
Assume that $h_L(\rho)$ is given by Equation \eqref{GFF_L}, and that $\psi:\tilde{D}\rightarrow \mathbb{D}$ is a conformal change of domain. Then the Liouville measure with log-singularities $(\bm\alpha,\bm z)$ and $(\bm\beta,\bm s)$ and background metric $|\psi'|^2\rho\circ\psi$ on $\tilde{D}$ has same law as the GMC measure on $\tilde D$ defined with the GFF 
\[
h_L(\rho)\circ\psi+Q\log |\psi'|.
\]
\end{prop}
Again we stress that this rule for changing domains is consistent with M\"obius transforms of the disk (see \cite[Proposition 3.7]{HRV16}).

The definition provided above for the unit boundary length quantum disk has been done in terms of the GMC determined by $h_{L(\rho_0)}$ under the weighted probability measure  \\
$d\mathbb{P}_{\rho_0}:=\frac{\nu_{h_L(\rho_0)}(\partial D)^{-2s/\gamma}}{\expect{\nu_{h_L(\rho_0)}(\partial D)^{-2s/\gamma}}}d\mathbb{P}$, where $\rho_0$ was the uniform probability measure on the disk.
The following statement shows that the choice made to define $h_{L(\rho)}$ is actually the good one when considering a change of background measure: 
\begin{prop}[Change of background measure]\label{change}
\hspace{0.5cm}\\
The random variables $h_{L(\rho)}-\frac{2}{\gamma} \log \nu_{h_L(\rho)}(\partial\mathbb{D})$ under the weighted probability measures $d\mathbb{P}_{\rho}=\frac{\nu_{h_L(\rho)}(\partial\mathbb{D})^{- 2s/\gamma}}{\expect{\nu_{h_L(\rho)}(\partial \mathbb{D})^{-2s/\gamma}}}d\mathbb{P}$ 
 have same law for any $\rho$ in $\mathbb{P}_{\mathbb{D}}$, provided that $\expect{\nu_{h_L(\rho)}(\partial\mathbb{D})^{-2s/\gamma}}<\infty$.\\
\end{prop}
\begin{proof}
Let $\rho\in\mathbb{P}_{\mathbb{D}}$ and write $h_{L(\rho)}=h_{L(\rho_0)}-(h,\rho-\rho_0)+ g_{\rho}$, where 
\[g_{\rho}(z):=Q m_{\rho}(G_\mathbb{D}(z,\cdot))+\sum_{i=1}^n \alpha_i (G_\mathbb{D}^\rho(z,z_i)-G_\mathbb{D}(z,z_i)) + \sum_{j=1}^r \frac{\beta_j}{2} (G_\mathbb{D}^\rho(z,s_j)-G_\mathbb{D}(z,s_j)).
\] 
Then $h_{L(\rho)}-\frac{2}{\gamma} \log \nu_{h_L(\rho)}(\partial\mathbb{D})$ under $\mathbb{P}_{\rho}$ has same law as  $h_{L(\rho_0)} + g_{\rho}-\frac{2}{\gamma} \log \nu_{h_{L(\rho_0)} + g_{\rho}}(\partial\mathbb{D})$ under the probability measure $d\mathbb{\tilde{P}}_{\rho_0}\propto \exp(s(h,\rho-\rho_0))d\mathbb{P}_{\rho_0}$.\\
We assume for now that $h_{L(\rho)}$ has been defined with the property that \[s(h,\rho-\rho_0)=-(h,g_{\rho})_{\nabla}.\] 
Therefore by the Cameron-Martin theorem this is tantamount to shifting the law of $h_{L(\rho_0)}$ by an additive term $-g_{\rho}$, which allows us to conclude the proof.

Now let us turn to the proof of the claim. For any $\phi$ in $H_a(D)$ we have:
\[2\pi\left(\phi,m_{\rho}(G_\mathbb{D}(z,\cdot))\right)_{\nabla}=\left(-\Delta \phi, m_{\rho}(G_\mathbb{D}(z,\cdot)\right)+\int_{\partial \mathbb{D}} \partial_n \phi(y) m_{\rho}(G_\mathbb{D}(\cdot,y))\lambda_{\partial}(dy)\] 
by the integration by parts formula, and then using Fubini identity and the property \eqref{IPP} of the Green's function yields 
\begin{align*}
    &=\left(\rho,2\pi (\phi(\cdot)-m_{\partial \mathbb{D}}(\phi))-\int_{\partial \mathbb{D}} G_\mathbb{D}(\cdot,y) \partial_{n} \phi(y)\lambda_{\partial}(dy)\right)+\int_{\partial \mathbb{D}} \partial_n \phi(y) m_{\rho}(G_\mathbb{D}(\cdot,y))\lambda_{\partial}(dy) \\
&=2\pi(\phi,\rho-\rho_0).
\end{align*}
Since $-s=Q-\sum \alpha_i -\sum \beta_j/2$, summing the corresponding terms in $g_{\rho}$ yields the result.
\end{proof}
In the sequel we may often simply write $h_L$ for $h_{L(\rho)}$ regardless of $\rho$ since we will work under such probability measures. Along the same lines one can see that the same reasoning remains valid when we consider $\rho$ in $\mathbb P_{\mathbb{D},\partial}$.

It is also interesting to note that such a statement should stay true in a broader context, \textit{e.g.} if we replace $\mathbb{D}$ by some Riemann surface, provided that we have defined $s$ consistently and considered the appropriate constant in front of the log.

\subsection{The unit boundary length quantum disk in the mating-of-trees approach}\label{def_DMS}
In this subsection we present the approach developed by Duplantier, Miller and Sheffield in the article \cite{DMS14} to define the unit boundary length quantum disk as an equivalence class of \textit{random surfaces}. This approach follows the ideas sketched by Sheffield in \cite{S16}.
Like before, we fix a constant $\gamma\in(0,2)$ throughout this subsection.

\subsubsection{Definition of the unit boundary length quantum disk}\label{DefBes}
We have already defined in the introduction the notion of \textit{quantum surface with $k$ marked points} as a class equivalence of $(k+2)$-tuples. Using this notion, we are now ready to introduce the definition of the unit boundary length quantum disk provided in \cite{DMS14}, which relies on an encoding using Bessel processes (see \cite{RY91}, \cite{GJY} or \cite[Subsection 3.2]{DMS14} for details on these objects).\\
The sketch of this encoding is the following one:\begin{enumerate}
\item We work in the strip $\mathcal{S}:=\mathbb{R}\times [0,\pi]$, that comes along with two marked points located at $\pm\infty$.
\item One then decomposes its Sobolev space into a Hilbertian sum of radial functions and functions with zero-mean on each vertical line. Thus describing a distribution on $\mathcal{S}$ is tantamount to giving its components on these two spaces.
\item The radial part of the field, which is nothing but a $\mathbb R\to\mathbb R$ map, is set to $\frac{2}{\gamma}\log e$, where $e$ is a Bessel excursion of dimension $3-\frac{4}{\gamma^2}$ parametrised so that $\frac{2}{\gamma}\log e$ has quadratic variation $2 du$.
\item The zero-vertical-mean part is given by the projection of an independent free boundary GFF.
\end{enumerate}
Let us denote the corresponding law of the field by $\mathcal{M}$. The \textit{unit boundary length quantum disk with two boundary marked points} is then defined by considering the law on quantum surfaces described by $(\mathcal{S},\mathfrak{h},-\infty,+\infty)$, where $\mathfrak{h}$ has the law of $\mathcal{M}$ conditioned on the event that $\nu_{\mathfrak{h}}(\partial\mathcal{S})=1$. Eventually we define the \textit{unit boundary length quantum disk with three marked points} to be the law on quantum surfaces described by $(\mathcal{S},\mathfrak{h},\omega,-\infty,+\infty)$, where $\omega$ is sampled according to $\nu_{\mathfrak{h}}$, which is the (random) boundary measure obtained from the field associated with the unit boundary length quantum disk with two marked points. In the sequel we will denote by $(\mu^{UBL}_{DMS},\nu^{UBL}_{DMS})$ the pair of random measures given by embedding the unit boundary length quantum disk with three marked points in $\mathbb{H}$ so that the three marked points are $(0,1,\infty)$.

In the definition of the quantum disk with two marked points, we see that we still have one degree of M\"obius freedom given by the horizontal translation along the real line. As a consequence we will say that we fix an embedding for the quantum surface when we choose an horizontal shift and consider the law of the Liouville Quantum Gravity measures obtained by taking the representative of $(\mathcal{S},\mathfrak{h},-\infty,+\infty)$ according to this translation. It will be convenient in the sequel to work in the \textit{maximal embedding}, where the maximum of the radial part of the field is attained at zero.

A similar definition can also be found in \cite{DMS14} for the unit area quantum sphere.

\subsubsection{Another construction for the unit boundary length quantum disk} 
We start by providing important limiting procedures that provide a construction of the unit boundary length quantum disk in terms of limits of GFF. For the sake of completeness, we provide here a sketch of the limiting procedure given in \cite[Proposition A.1]{DMS14}: 
\begin{itemize}
\item For positive $C$ and $\varepsilon$, let $h_{C,\varepsilon}$ be a GFF on $\mathbb{D}\cap\mathbb{H}$ with zero (resp. free) boundary conditions on $\partial \mathbb{D}\cap\mathbb{H}$ (resp. $L:=\mathbb{D}\cap\partial\mathbb{H}$) conditioned on $\lbrace\nu_{h_{C,\varepsilon}}(L)\in[e^{C},e^{C}(1+\varepsilon)]\rbrace$.
\item Sample $w$ from the boundary measure $\nu_{h_{C,\varepsilon}}$ and conformally map $(\mathbb{H},w,\infty)$ to $(\mathcal{S},\infty,-\infty)$. There is one extra degree of freedom---that is horizontal translations---when fixing the embedding: a convenient choice can be made\footnote{Without loss of generality one may instead choose to fix the horizontal translation using the maximal embedding, which we will assume in the sequel.} by assuming that $\nu_{\mathfrak h_{C,\varepsilon}}(\{0\}\times[0,\pi])=\frac{e^C}2$, where $\mathfrak{h}_{C,\varepsilon}$ is the field obtained on $\mathcal{S}$ using the usual rule of change of coordinates.
\item The law of the field $\mathfrak{h}_{C,\varepsilon}-\frac{2}{\gamma} C$ then converges weakly in the space of distributions to that of the unit boundary length quantum disk with two marked points $(\infty,-\infty)$ on $\mathcal{S}$ as $C\rightarrow\infty$ and then $\varepsilon\rightarrow0$, when embedded according to the chosen translation.
\end{itemize}
Note that we can give a meaning to a distribution on $D$ given $h_0$ defined on a subdomain $D_0$ of $D$ by extending it to zero on the complementary of $D_0$ in $D$: this justifies the previous statements. The same applies for the Liouville measures associated to the field $h_0$, which we set to be equal to zero outside of the domain $D_0$. In the sequel we will implicitly make use of this convention. We will also say that a sequence of random measures $(\mu_\varepsilon,\nu_\varepsilon)$ converges weakly in law when the law of the pair of random measures weakly converges in the sense of Radon measures. We will make use of the same terminology with fields by implying that weak convergence in law simply means weak convergence (in the sense of distributions) of the law of a sequence of fields. 

The following procedure for defining the law of the unit boundary length quantum disk, slightly more explicit, can be raised from the proof of \cite[Proposition A.1]{DMS14}
\begin{prop}[Limiting procedure for the unit boundary length quantum disk with two marked points]\label{ADMS1}
\hspace{0cm}\\
Let $h_0$ be a GFF on $\mathbb{D}\cap\mathbb{H}$ with zero (resp. free) boundary conditions on $\partial \mathbb{D}\cap\mathbb{H}$ (resp. $\mathbb{D}\cap\partial\mathbb{H}$) and assume that $C_{\varepsilon}$ goes to $+\infty$ as $\varepsilon$ goes to zero. Define a field
\begin{center}
$ h^{\varepsilon}:=h_0 - \gamma \log|z|-C_{\varepsilon}$
\end{center} 
and condition on the event $\lbrace\nu_{h^{\varepsilon}}(L)\in[e^{-\gamma\delta},e^{\gamma\delta}]\rbrace$. Denote by $\psi$ the unique conformal mapping between $\mathbb{H}$ and $\mathcal{S}$ that sends $(0,\infty)$ to $(-\infty,+\infty)$ and such that the maximum of the radial part of the field \[
\mathfrak h^\varepsilon:= h^{\varepsilon}\circ\psi+Q\ln\vert\psi'\vert
\]
is attained at $0$.

Then, as $\varepsilon\rightarrow0$ and then $\delta\rightarrow0$, the field $\mathfrak h^\varepsilon$ converges weakly in law to $\mathfrak h$, whose law is given by $\mathcal{M}$ conditioned on the event that $\nu_{\mathfrak{h}}(\partial\mathcal{S})=1$ and embedded in $(\mathcal{S},-\infty,-\infty)$ according to the maximal embedding.
\end{prop}
\begin{proof}
The law of the quantum surface $(\mathbb{H},h^{\varepsilon},0,\infty)$ is the same as the law of the quantum surface $(\mathcal{S},\mathfrak{h}^{\varepsilon},+\infty, -\infty)$ using the conformal mapping $\psi$, where 
\[\mathfrak{h}^{\varepsilon}:=\mathfrak{h}_0 + (\gamma-Q)\mathrm{Re}(\cdot)-C_{\varepsilon}\]
and $\mathfrak{h}_0$ is a GFF on $\mathcal{S}_+=[0,+\infty)\times [0,i\pi]$ with zero (resp. free) boundary conditions on $\lbrace0\rbrace\times[0,i\pi]$ (resp. $(0,+\infty)\times \lbrace0,i\pi\rbrace$). We are precisely in the setting of the proof of \cite[Proposition A.1]{DMS14} establishing convergence of the above sequence to the unit boundary length quantum disk with two marked points.
\end{proof}
We now give the following similar approximation proposition, which describes the limiting procedure for the unit boundary length quantum disk with three marked points that we will work with in the rest of the article. For future convenience we introduce for positive $\varepsilon$ the domain $D^{\varepsilon}:=\frac{1}{\sqrt{\varepsilon}}\mathbb{D}\cap\mathbb{H}$ and let $h_0^{\varepsilon}$ be a GFF on $D^{\varepsilon}$ with zero (resp. free) boundary conditions on $\partial D^{\varepsilon}\cap\mathbb{H}$ (resp. $\partial D^{\varepsilon}\cap\partial\mathbb{H}$). We also let $G_{D^{\varepsilon}}$ be the Green's kernel associated to this problem.
\begin{prop}[Limiting procedure for the unit boundary length quantum disk with three marked points]\label{ADMS1bis}
\hspace{0cm}\\
Let $h^{\varepsilon}$ be a field on $D^{\varepsilon}$ defined by 
\[h^{\varepsilon}:=h_0^{\varepsilon}+\frac{1}{2}(2Q-\gamma)\log\varepsilon + \frac{\gamma}{2} G_{D^{\varepsilon}}(z,0).\]
Denote by $\hat{h}^\varepsilon$ the field whose law is given by conditionning $h^\varepsilon$ on the event $E^{\varepsilon}_{\delta}(\partial\mathbb{H}):=\lbrace\nu_{h^{\varepsilon}}(\partial\mathbb{H})\in[e^{-\gamma\delta},e^{\gamma\delta}]\rbrace$, sampling $\omega^{\varepsilon}$ on $\frac{1}{\sqrt{\varepsilon}}\mathbb{D}\cap\partial\mathbb{H}$ according to the law of $\nu_{h^{\varepsilon}}$ and conformally mapping $(0,\omega^{\varepsilon},\infty)$ to $(0,1,\infty)$ with the M\"obius transform of $\mathbb{H}$: $z\mapsto \omega^\varepsilon z$. 

Then, when we let $\varepsilon\rightarrow0$ and then $\delta\rightarrow0$, the pair of random measures $(\mu_{\hat h^{\varepsilon}},\nu_{\hat h^{\varepsilon}})$ converges weakly in law to the pair of random measures $(\mu^{UBL}_{DMS},\nu^{UBL}_{DMS})$.
\end{prop}
Before dealing with the proof, we shed light on an useful scaling property of the Green's function $G_{D^{\varepsilon}}$. By the reflection principle one has: 
\begin{equation}\label{Green_scale}
G_{D^{\varepsilon}}(x,y)=-\log\frac{|x-y||x-y^*|}{|\frac{1}{\varepsilon}-xy||\frac{1}{\varepsilon}-xy^*|} + \log\varepsilon= G_{\mathbb{H}}(x,y)-\log \varepsilon+r_{\varepsilon}(x,y)
\end{equation}
where for any $y\in\mathbb{H}$, $r_{\varepsilon}(\cdot,y)$ is harmonic and converges uniformly on every compact to zero as $\varepsilon\rightarrow0$.

\begin{proof}
When we apply the conformal map $z\mapsto \sqrt{\varepsilon}z$ on $D^{\varepsilon}$ the law of the pushforwarded Liouville measures are the same as the ones on $\mathbb{D}\cap\mathbb{H}$ given by the field
\[h_0^{\varepsilon=1}+\frac{1}{2}(Q-\gamma)\log\varepsilon - \gamma \log|z|\]
since $G_{D^{\varepsilon}}(z,0)=-2\log|z| - \log\varepsilon$, and that by conformal invariance of the GFF $h_0^{\varepsilon=1}$ and $h_0^{\varepsilon}(\sqrt{\varepsilon}\cdot)$ have same law.
Letting $C_{\varepsilon}$ be $\frac{1}{2}(Q-\gamma)\log\frac1\varepsilon$ (which goes to $+\infty$ as $\varepsilon$ goes to zero since $Q-\gamma>0$), the previous result yields that (when rescaled via the maximal embedding) $(\mu_{h^{\varepsilon}},\nu_{h^{\varepsilon}})$ converges weakly in law to to the pair of random measures $(\mu_{DMS},\nu_{DMS})$ given by mapping conformally the unit boundary length quantum disk with two marked points into $\mathbb{H}$ and with the horizontal translation fixed by the maximal embedding.
Therefore under this maximal embedding, if we work on $\mathcal{S}$ and sample $\omega^{\varepsilon}$ under $\nu_{\mathfrak h^{\varepsilon}}$ and likewise sample $\omega$ according to $\nu_{DMS}$, we can find a coupling (thanks to Skorokhod's representation theorem) between these variables such that the measures converge almost surely and $\lim\limits_{\varepsilon\rightarrow0} |w^{\varepsilon}- w|=0$ in probability. Hence conformally mapping $\omega^{\varepsilon}$ to 1 and taking the limit gives the law of the unit boundary length quantum disk embedded in $\mathbb{H}$ so that the three marked points are $(0,1,\infty)$. 
\end{proof}

Eventually, we will need to know some information on the location of the point sampled. In the following statement we set $\mathfrak h_0^{\varepsilon}$ to be a GFF on $\mathcal D^{\varepsilon}:=(\frac{1}{2}\log \varepsilon,\infty)\times [0,i\pi]$ with zero (resp. free) boundary conditions on $\lbrace\frac{1}{2}\log\varepsilon\rbrace\times [0,i\pi]$ (resp. $(\frac{1}{2}\log\varepsilon,\infty)\times \lbrace0,i\pi\rbrace$). We also introduce the field $\mathfrak h^{\varepsilon}:=\mathfrak h_0^{\varepsilon}-(Q-\gamma)\mathrm{Re}(z)+(Q-\gamma)\log\varepsilon$, which we extend as explained above to $0$ outside of the domain $\mathcal D^{\varepsilon}$. Similarly the associated measures are extended to zero outside of $\mathcal D^{\varepsilon}$.
\begin{lemma}[Useful estimates]\label{LOCS}
\hspace{0cm}\\
Sample $w^{\varepsilon}$ according to $\nu_{\mathfrak h^{\varepsilon}}$ and denote $\mathcal E^{\varepsilon}_{\delta}(\partial\mathcal S)$ the event that $\lbrace\nu_{\mathfrak h^{\varepsilon}}(\partial\mathcal S)\in[e^{-\gamma\delta},e^{\gamma\delta}]\rbrace$.
Then we have the following estimates: 
\begin{itemize}
\item Conditionally on $\mathcal E^{\varepsilon}_{\delta}(\partial\mathcal S)$, $\lim\limits_{\varepsilon\rightarrow0}\frac{\mathrm{Re}(w^{\varepsilon})}{|\log\varepsilon|^{2/3}}=0$ in law.
\item $\lim\limits_{\delta\rightarrow 0} \lim\limits_{\varepsilon\rightarrow 0}  \hat{\mathbb{P}}^{\varepsilon}(\mathcal H^{\varepsilon}\vert \mathcal E^{\varepsilon}_{\delta}(\partial\mathbb{H}))=1$,
\end{itemize}
where $\mathcal H^{\varepsilon}:=\lbrace \mathcal A_{\varepsilon}\geq -|\log \varepsilon|^{2/3}\rbrace$, with $\mathcal A_{\varepsilon}$ being the  mean value of $\mathfrak h^{\varepsilon}$ on $\lbrace \mathrm{Re}(w^{\varepsilon})\rbrace\times [0,i\pi]$.
\end{lemma}
\begin{proof}
We use the radial/angular decomposition of the GFF to write it under the form $\mathfrak h^{\varepsilon}=\mathfrak h^{\varepsilon}_{\text{rad}}(\mathrm{Re}(z))+\mathfrak h^{\varepsilon}_{\text{ang}}(z)$ where $\mathfrak h^{\varepsilon}_{\text{rad}}:\mathbb{R}\to\mathbb{R}$ is the mean value of the field on the line in $\mathcal{S}$ with real part $\mathrm{Re}(z)$. From the covariance kernel of the GFF we know (see \cite[Lemma A.3]{DMS14}) that this radial component of $\mathfrak h^{\varepsilon}$, for $t\geq\frac12\log\varepsilon$, has same law as $B_{2t}-(Q-\gamma)t+(Q-\gamma)\log\varepsilon$, where $(B_t)_{t\geq\log\varepsilon}$ is a Brownian motion with $B_{\log\varepsilon}=0$. Denote by $L^{\varepsilon}$ the (first) location where $B_{2t}-(Q-\gamma)t$ achieves its maximum. According to the proof of the previous result, we have that the sequence $L^{\varepsilon}-\mathrm{Re}(w^{\varepsilon})$ is tight, so in order to get the result it suffices to prove that $\lim\limits_{\varepsilon\rightarrow0}\frac{L^{\varepsilon}}{|\log\varepsilon|^{2/3}}=0$ in law (this is precisely the reason why we have chosen to work in the maximal embedding).

We use the notations of \cite[Lemma A.4]{DMS14} and define the event $F^{\varepsilon}_{C}$ that the maximum of $B_{2t}-(Q-\gamma)t+(Q-\gamma)\log\varepsilon$ is larger than $-C$, with the properties that (\cite[Lemma A.4]{DMS14}) for any positive $\delta$ and uniformly in $\varepsilon$,
\[\lim\limits_{C\rightarrow + \infty} \mathbb{P}(F^{\varepsilon}_{C}\vert E^{\varepsilon}_{\delta})=1\quad\text{and}\quad\mathbb{P}(E^{\varepsilon}_{\delta}\vert F^{\varepsilon}_{C})>0.\]
It is therefore enough to show that the result holds for 
\[T^{\varepsilon}:=\inf \lbrace t\geq \frac{1}{2}\log \varepsilon, B_{2t}-(Q-\gamma)t\geq -C-(Q-\gamma)\log\varepsilon\rbrace
\] for any fixed $C$, when conditioning on $F^{\varepsilon}_{C}$. By the Markov property for the Brownian motion, $T^{\varepsilon}$ has same law as 
\[
\frac{1}{2}\left( \log\varepsilon + S_{A}\right):=\frac{1}{2}\left( \log\varepsilon + \inf \lbrace t\geq 0, \tilde{B}_{t}-at\geq A\rbrace\right),
\] where $a=\frac{Q-\gamma}{2}$, $A=-\frac{Q-\gamma}{2}\log\varepsilon-C$ and $\tilde{B}$ is a standard Brownian motion. The result then follows from \cite[Lemma 4.5]{AHS16}: conditioned on $S_A<\infty$, we have
\[\lim\limits_{A\rightarrow +\infty} \frac{S_A-a^{-1}A}{A^{2/3}}=0.\]
Thanks to this point we now have that \begin{align*}
    \hat{\mathbb{P}}^{\varepsilon}(H^{\varepsilon}\vert E^{\varepsilon}_{\delta}(\partial\mathbb{H}))    &=\hat{\mathbb{P}}^{\varepsilon}(B_{2 \mathrm{Re}(w^{\varepsilon})}\geq (Q-\gamma)\mathrm{Re}(w^{\varepsilon}) - |\log\varepsilon|^{2/3}\vert E^{\varepsilon}_{\delta}(\partial\mathbb{H}))+o(1)\\
    &=\hat{\mathbb{P}}^{\varepsilon}(B_{2 L^{\varepsilon}}\geq (Q-\gamma)L^{\varepsilon} - |\log\varepsilon|^{2/3}\vert E^{\varepsilon}_{\delta}(\partial\mathbb{H}))+o(1),
\end{align*}
which tends to 1 according to \cite[Lemma A.4]{DMS14}.
\end{proof}

\section{A limiting procedure for the unit boundary length quantum disk with three log-singularities}\label{equal}

In Proposition~\ref{ADMS1bis} we have studied two alternative definitions for the unit boundary length quantum disk and shed light on a procedure giving in the limit the law of one of them: the unit boundary length quantum disk with three marked points $(\mu^{UBL}_{DMS},\nu^{UBL}_{DMS})$.\\
The goal of this section is to show that we can slightly change this scheme to provide similarly a limiting procedure for the other definition, that is the unit boundary length quantum disk with three log-singularities $(\gamma,\gamma,\gamma)$ located at $(-1,1,-i)$ that we denoted $(\mu^{UBL}_{HRV},\nu^{UBL}_{HRV})$. Eventually we will show that this change in the scheme becomes negligible in the limit, which will yield the equality in law of the two objects previously exposed.

\subsection{Perturbation of the previous scheme}
Let us start with the limiting procedure obtained in Proposition~\ref{ADMS1bis}: \begin{itemize}
\item We first considered the field $h^{\varepsilon}$ on $D^{\varepsilon}$.
\item We conditioned on the event $E^{\varepsilon}_{\delta}(\partial\mathbb{H})=\lbrace\nu_{h^{\varepsilon}}(\partial\mathbb{H})\in[e^{-\gamma\delta},e^{\gamma\delta}]\rbrace$; let $\mathbb{P}^{\varepsilon}_{\delta}$ be the law of the conditioned field.  
\item We sampled $w^{\varepsilon}$ on $\frac{1}{\sqrt{\varepsilon}}\mathbb{D}\cap\partial\mathbb{H}$ according to $\nu_{h^{\varepsilon}}$ and sent the $3$-uple $(0,w^\varepsilon,\infty)$ to $(0,1,\infty)$ via a M\"obius transform of $\mathbb{H}$, where $h^{\varepsilon}$ has law $\mathbb{P}^{\varepsilon}_{\delta}$.
\item We let $\varepsilon\rightarrow0$ and then $\delta\rightarrow0$.
\end{itemize}
Our goal in this section would be to show that for any $F$ non-negative bounded continuous (with the topology of weak convergence) functional on the space of measures over $\overline{\mathbb{H}}$
\begin{equation}\label{eq:convergence}
\lim\limits_{\delta\rightarrow 0} \lim\limits_{\varepsilon\rightarrow 0}  \mathbb{E}\left[F(\mu_{h^{\varepsilon}}, \nu_{h^{\varepsilon}})\vert E^{\varepsilon}_{\delta}(\partial\mathbb{H})\right]=\frac{\mathbb{E}\left[F(\frac{\mu_{h_L}}{\nu_{h_L}(\partial\mathbb{H})^2}, \frac{\nu_{h_L}}{\nu_{h_L}(\partial\mathbb{H})})\nu_{h_L}(\partial\mathbb{H})^\frac{2Q-3\gamma}{\gamma}\right]}{\mathbb{E}\left[\nu_{h_L}(\partial\mathbb{H})^\frac{2Q-3\gamma}{\gamma}\right]},
\end{equation}
which is the expression defining the law of $(\mu^{UBL}_{HRV},\nu^{UBL}_{HRV})$.
However we will start by proving this result when considering the modified scheme which consists in sampling $\hat{w}^{\varepsilon}$ under the weighted probability measure $\hat{\mathbb{P}}^{\varepsilon}_{\delta}$ defined by $d\mathbb{\hat{P}}^{\varepsilon}_{\delta}\propto\nu_{h_{\varepsilon}}(\partial\mathbb{H})d\mathbb{P}^{\varepsilon}_{\delta}$ rather than $\mathbb{P}^{\varepsilon}_{\delta}$. We will prove in Subsection~\ref{subsection:final} that this is enough to prove that Equation~\eqref{eq:convergence} holds true.

\subsubsection{Change induced by the perturbation of the procedure}

In order to study how this modification of the scheme affects the law of the random measures, we first recall some basic properties of \textit{rooted measures}, whose goal is to consider the law of the pair $(h,w)$ where $h$ is a distribution on $D$ and $w$ in $\partial D$ under the probability measure $\propto \nu_{h}(\partial D) d\mathbb{P}$. Studying the marginal and conditional laws of the two variables, the authors in \cite[Lemma A.7]{DMS14} proved that a sample from the weighted law $\propto\nu_{h^{\varepsilon}}(\partial\mathbb{H})d\mathbb{P}^\varepsilon$ can be produced by: \begin{itemize}
\item First sampling $h^\varepsilon$ according to its unweighted law.
\item Picking $w$ independently of $h^\varepsilon$ according to its marginal law and then shifting $h^\varepsilon$ by an additional factor $\frac{\gamma}{2} G_{D^\varepsilon}(\cdot,w)$, where $G_{D^\varepsilon}$ is the Green's kernel with zero (resp.free) boundary conditions on $\partial D^{\varepsilon}\cap\mathbb{H}$ (resp. $\partial D^{\varepsilon}\cap\partial\mathbb{H}$).
\end{itemize}
With the same arguments as in the proof of \cite[Lemma A.7]{DMS14}, we can show similarly that a sampling from the law $\mathbb{\hat{P}}^{\varepsilon}_{\delta}$ (that is with conditioning on the boundary length) can be produced by: 
\begin{itemize}
\item First sampling $h^\varepsilon$ according to its unweighted law (that is without conditioning the boundary length).
\item Picking $\hat w$ independently of $h^\varepsilon$ according to its marginal law and then shifting $h^\varepsilon$ by an additional factor $\frac{\gamma}{2} G_{D^\varepsilon}(\cdot,\hat w)$.
\item Conditioning on the event $E^{\varepsilon}_{\delta}(\partial\mathbb{H})$.
\end{itemize}
We explain in the proof of Theorem \ref{theorem_three} that by doing so (and since we condition on the event that $\nu_{h^{\varepsilon}}(\partial\mathbb{H})$ tends to $1$ as $\delta\rightarrow0$) we still get the same result in the limit. As a consequence in the rest of the document, instead of working with $h_0^{\varepsilon}$ being defined according to the unweighted law of the GFF $\mathbb{P}^{\varepsilon}_{\delta}$, we work with the field\footnote{We still denote this field by $h^\varepsilon$ in order not to overload the notations, but this field differs from the one considered in the previous section since is picked according to the weighted measure $\mathbb{\hat{P}}^{\varepsilon}_{\delta}$.} on $D^{\varepsilon}$ given by
\[
h^{\varepsilon}:= h_0^{\varepsilon}+\frac{\gamma}{2} G_{D^\varepsilon}(z,0) + \frac{1}{2}(2Q-\gamma)\log \varepsilon
\]
as described in the above procedure, and where $h_0^{\varepsilon}$ has the law of $\mathbb{\hat{P}}^{\varepsilon}_{\delta}$, that is a GFF on $D^{\varepsilon}$ with zero (resp. free) boundary conditions on $\partial D^{\varepsilon}\cap\mathbb{H}$ (resp. $\partial D^{\varepsilon}\cap\partial\mathbb{H}$) but \textbf{sampled under the weighted law $\propto\nu_{h^{\varepsilon}}(\partial\mathbb{H})d\mathbb{P}_\delta^{\varepsilon}$}. 
We also introduce $h_L$, the GFF in $\mathbb{H}$ with log-singularities $(\gamma,\gamma,\gamma)$ at $(0,1,\infty)$, by considering as background measure $c$ the uniform probability measure on the semi-circle $\partial\mathbb{D}\cap\mathbb{H}$:
\begin{equation}\label{GFF_semi_circle}
h_L(z):=h_c(z) +\frac{\gamma}{2}(G_\mathbb{H}(z,0)+G_\mathbb{H}(z,1))+s\log\left(|z|\vee1\right)+\tilde{C},
\end{equation}
where recall that $s$ was defined in Equation~\eqref{def:s} and is given here by 
\[
s=\frac{3\gamma}{2}-Q,\]
and where $\tilde C$ is a constant (whose value is not relevant for our purpose) chosen so that the field $h_L$ has zero-mean under the background measure $c$.
To define this one can start with the GFF with log-singularities $(\gamma,\gamma,\gamma)$ at $(-1,-i,1)$ and background measure $l(z):=\frac{1}{\pi}\frac{2}{|1-z|^2}\mathds{1}_{z\in[-i,i]}$ in $\mathbb{D}$:
\[
h_l -s m_{l}(G_\mathbb{D}(z,\cdot)) + \frac{\gamma}{2}(G_\mathbb{D}(z,-1)+G_\mathbb{D}(z,-i)+G_\mathbb{D}(z,1))+C,
\]
and then apply the rule for changing domains \ref{CCD} with the conformal transformations between $\mathbb{D}$ and $\mathbb{H}$ given by 
\[\phi^{-1}(z):=i\frac{1+z}{1-z}\quad\text{with inverse}\quad\phi(z)=\frac{z-i}{z+i}.
\]
Again $C$ is such that the above field has zero-mean under the background measure $l$, and we have used Equations~\eqref{def:green_rho} and~\eqref{GFF_L} to write down the expression of the field on $\mathbb{D}$.
To see why one can go from this field to the one described in Equation~\eqref{GFF_semi_circle}, note that these conformal mappings are such that for any $y\neq 1$,
\[G_{\mathbb{D}}(\phi(z),y)=G_{\mathbb{H}}(z,\phi^{-1}(y))+2\log |z+i|-2\log|y-1| \quad\text{and}\quad \log|\phi'(z)|=-2\log|z+i|.
\]
As a consequence and using the fact that $\frac{1}{\pi}\int_{-1}^{1} G_{\mathbb{H}}(z,\phi^{-1}(y))|(\phi^{-1})'(y)|^2 dy=-\log\left(|z|\vee1\right)$ we recover the expression~\eqref{GFF_semi_circle}.
In the sequel, we will use the shorthand $g(z):=-\log\left(|z|\vee1\right)$.

We are now ready to quantify how the modification of the scheme affects the law of the random measures:
\begin{prop}[Approximation by sampling]\label{Sampling1}
\hspace{0.5cm}\\
Let $\hat{w}^{\varepsilon}$ be a sample on $\frac{1}{\sqrt{\varepsilon}}\mathbb{D}\cap\partial\mathbb{H}$ from $\nu_{h^{\varepsilon}}$, set $d_\varepsilon:=\frac{1}{|\hat{w}^{\varepsilon}|^2\varepsilon}$ and conformally map $(\mathbb{H},0,\hat w^{\varepsilon},\infty)$ to $(\mathbb{H},0,1,\infty)$\footnote{When $\hat{w}^\varepsilon$ is negative it would be more precise to say that we conformally map $(\mathbb{H},\hat w^{\varepsilon},0,\infty)$ to $(\mathbb{H},0,1,\infty)$}. 
Then the quantum surface thus defined has the law of $(\mathbb{H},\hat{h}^\varepsilon,0,1,\infty)$, where 
\[\hat{h}^{\varepsilon}:=\hat h_L^{\varepsilon}+A_{\varepsilon}\frac{g(z)}{\log d_{\varepsilon}}+A_{\varepsilon} + r_{\varepsilon}(z).
\]
In this expression $\hat h_L^{\varepsilon}$ is defined by Equation \eqref{GFF_semi_circle} with Dirichlet boundary conditions in $\hat D^\varepsilon$\footnote{More precisely one needs to consider $h_c^\varepsilon$ to be a GFF in $\hat D^\varepsilon:=D^{d_\varepsilon}$ with with zero (resp. free) boundary conditions on $\partial \hat D^\varepsilon\cap\mathbb{H}$ (resp. $\partial \hat D^\varepsilon\cap\partial\mathbb{H}$) conditioned to have zero mean on $\partial\mathbb{D}\cap\mathbb{H}$ instead of $h_c$.}, $A_{\varepsilon}$ is the mean value of $\hat{h}^{\varepsilon}-r_\varepsilon$ on the semi-circle $\partial\mathbb{D}\cap\mathbb{H}$, and $r_{\varepsilon}(z)$ is harmonic and such that the law of $\frac{r_{\varepsilon}}{|\log\varepsilon|^{2/3}}$ converges (in law) uniformly to zero on every compact subset of $\mathbb{H}$.
\end{prop} 

\begin{proof}
Since we are working with a field $h_0^{\varepsilon}$ considered under the weighted measure $\mathbb{\hat{P}}^{\varepsilon}_{\delta}$, we know that sampling $\hat{w}^{\varepsilon}$ is tantamount to considering the law of the Liouville measure of the field
 \[
 h_0^{\varepsilon}+\frac{\gamma}{2} G_{D^\varepsilon}(z,0) +\frac{\gamma}{2} G_{D^\varepsilon}(z,\hat{w}^{\varepsilon}) + \frac{2Q-\gamma}{2} \log \varepsilon
 \]
where $\hat{w}^{\varepsilon}$ is chosen independently of $h_0^{\varepsilon}$ from its marginal law. Applying the conformal mapping $\hat\psi:z\mapsto \hat{w}^{\varepsilon}z$ if $\hat{w}^\varepsilon>0$ and $\hat\psi:z\mapsto -\hat{w}^{\varepsilon}(z-1)$ otherwise, the quantum surface $(\mathbb H,h^\varepsilon,0,\hat{w}^{\varepsilon},\infty)$ is equivalent to $(\mathbb H,\hat h^\varepsilon,0,1,\infty)$ with
\[\hat{h}^{\varepsilon}:=\hat{h}_0^{\varepsilon}+\frac{\gamma}{2} (G_{D^\varepsilon}(z,0) + G_{D^\varepsilon}(z,1)) + \frac{2Q-\gamma}{2} \log \varepsilon + Q\log |\hat{\psi}'|\] \[ +\frac{\gamma}{2} \left(G_{D^\varepsilon}(\hat{\psi}(z),\hat{\psi}(0)) - G_{D^\varepsilon}(z,0) + G_{D^\varepsilon}(\hat{\psi}(z),\hat{\psi}(1))-G_{D^\varepsilon}(z,1)\right)\]
where $\hat{h}_0^{\varepsilon}$ is a GFF similar to $h_0^\varepsilon$ but defined on the domain $\hat D^\varepsilon$.\\
Using the scaling property for the Green's functions \eqref{Green_scale} and the change under M\"obius transform yields 
\[
\hat{h}_0^{\varepsilon}+\frac{\gamma}{2} \left(G_{\mathbb{H}}(z,0) + G_{\mathbb{H}}(z,1)\right) + \frac{1}{2}(2Q-3\gamma)\log \varepsilon + r_{\varepsilon}(z),
\] 
where
\[
r_{\varepsilon}(z):=\frac{\gamma}{2} \left(-4\log\vert\hat w^\varepsilon\vert -2\log\vert1-\varepsilon(\hat{w}^{\varepsilon})^2z\vert\right).
\]
We will show in Lemma \ref{LOCS2} that $r_{\varepsilon}$ satisfies the assumptions of the proposition.
Now by the orthogonal decomposition (Lemma \ref{OD2}) for the GFF we can write that
 \[
 \hat{h}_0^{\varepsilon}=\hat{h}_c^{\varepsilon}+\frac{\hat A_{\varepsilon}}{Var(\hat A_{\varepsilon})}\xi^{\varepsilon}(z)
 \]
where $\xi^{\varepsilon}(z)=-2\log\left(|z|\vee 1\right)+\log d_{\varepsilon}$ like in Lemma \ref{OD2}, $\hat A_{\varepsilon}:=(\hat h^\varepsilon_0,\rho_{\varepsilon})$ is Gaussian with mean zero and variance $\log d_{\varepsilon}$, and is independent of $\hat{h}_c^{\varepsilon}$ defined to have the law of $\hat{h}_0^{\varepsilon}$ conditioned to have zero mean on the semi-circle $\partial\mathbb{D}\cap\mathbb{H}$. 
As a consequence we get that 
\[
\hat{h}^{\varepsilon}=\hat{h}_{L}^{\varepsilon} + \left( \hat A_{\varepsilon}+s\log d_{\varepsilon}\right)\frac{\log\left(|z|\vee 1\right)}{\log d_{\varepsilon}} + \left(\hat A_{\varepsilon} + s\log d_{\varepsilon}\right) + r_{\varepsilon}(z).
\]
We will show in Lemma \ref{LOCS2} that $r_{\varepsilon}(z)$ is as desired, so to finish up with the proof, we note that
$A_{\varepsilon}:=\hat A_{\varepsilon}+s\log d_{\varepsilon}$ does indeed coincide with the mean value on the semi-circle $\partial\mathbb{D}\cap\mathbb{H}$ of the field $\hat{h}^{\varepsilon}$ (up to the remainder $r_\varepsilon$) since $\hat{h}_{L}^{\varepsilon}$ has zero mean on $\partial\mathbb{D}\cap\mathbb{H}$.
\end{proof}
\begin{rmk}\label{rmk:A}
From the proof of Proposition \ref{Sampling1} we shed light on the fact that $A_\varepsilon$ is Gaussian with mean $s\log d_{\varepsilon}$ and variance $\log d_{\varepsilon}$; besides it is independent of $h_L^{\varepsilon}$.
\end{rmk}

\subsubsection{Limiting law for the procedure}
Now that we have seen how this perturbation affects the law of the measures, we work with the expression obtained \[\hat{h}^{\varepsilon}=\hat{h}_L^{\varepsilon}+A_{\varepsilon}\frac{g(z)}{\log d_{\varepsilon}}+A_{\varepsilon} + r_{\varepsilon}(z)\]
and study the limiting object this field gives rise to. The following statement aims to prove that thanks to the conditioning on the boundary length, this field will converge weakly in law toward a weighted Gaussian field.

\begin{prop}[Limiting law for the field]\label{Approximation}
\hspace{0.5cm}\\
Let $h_L$ be defined by Equation \eqref{GFF_semi_circle}. Then there exists a sequence $(C^{\varepsilon}_{\delta})_{\varepsilon,\delta>0}$ such that for any $F$ non-negative bounded continuous (with the topology of weak convergence) functional on $\Hr^{-1}(\mathbb{H})$ which is invariant by adding a constant\footnote{This simply means that for any complex $C$ and field $h$, $F(h+C)=F(h)$.} we have \begin{equation}
\lim\limits_{\delta\rightarrow 0} \lim\limits_{\varepsilon\rightarrow 0}  C^{\varepsilon}_{\delta} \mathbb{E}\left[F(\hat h^{\varepsilon})\mathds{1}_{E^{\varepsilon}_{\delta}(\partial\mathbb{H})}\mathds{1}_{H^{\varepsilon}} \right]=\mathbb{E}\left[F(h_L)\nu_{h_L}^{-\frac{2s}{\gamma}}(\partial\mathbb{H})\right],
\end{equation}
where $H^{\varepsilon}:=\lbrace A_{\varepsilon}-s\log\varepsilon\geq - |\log\varepsilon|^{2/3}\rbrace$.
\end{prop}
\begin{proof}
Recall that in Remark~\ref{rmk:A} we described the law of $A_\varepsilon$. Therefore defining a probability measure by 
\[
d\mathbb{Q}^{\varepsilon}= \exp\left(-s(A_{\varepsilon}+s\log d_\varepsilon)\right)d\hat{\mathbb{P}}_\delta^{\varepsilon},
\]
by Girsanov theorem\footnote{Recall that $d_\varepsilon$ is defined via the point $\hat w^\varepsilon$, which was sampled according to its marginal law and therefore is independent from all the other random variables involved in the present proof.} under this new probability measure $A_{\varepsilon}$ remains Gaussian with same variance but mean zero, and $h^{\varepsilon}_{L}$ remains independent of $A_{\varepsilon}$ with same law as under $\hat{\mathbb{P}}_\delta^{\varepsilon}$.
Moreover since we condition on the event $\lbrace\nu_{\hat h_L^{\varepsilon}}(\partial\mathbb{H})\in[e^{-\gamma\delta},e^{\gamma\delta}]\rbrace$ one has
\[
\nu_{\hat h_L^{\varepsilon} + A_{\varepsilon}}(\partial\mathbb{H})^{-\frac{2s}\gamma}\exp(sA_{\varepsilon})=1+o(1),
\] so proving the result is tantamount to showing that \[\lim\limits_{\delta\rightarrow 0} \lim\limits_{\varepsilon\rightarrow 0}  C^{\varepsilon}_{\delta} \mathbb{E}^{\mathbb{Q}^{\varepsilon}}\left[F(\hat h^{\varepsilon})\nu_{\hat h_L^{\varepsilon}+\frac{A_{\varepsilon}g}{\log d_\varepsilon}+r_\varepsilon}(\partial\mathbb{H})^{-\frac {2s}\gamma}\mathds{1}_{E^{\varepsilon}_{\delta}(\partial\mathbb{H})}\mathds{1}_{H^{\varepsilon}} \right]=\mathbb{E}\left[F(h_L)\nu_{h_L}^{-\frac{2s}\gamma}(\partial\mathbb{H})\right].\]
Using Proposition \ref{Sampling1} and anticipating on Lemma \ref{CIU} we see that, without loss of generality, we can remove the terms $r_\varepsilon$ and $d_\varepsilon$ and only keep working with the field $h_L^{\varepsilon}+\frac{A_{\varepsilon}g}{-\log \varepsilon}$ instead of $\hat h_L^{\varepsilon}+\frac{A_{\varepsilon}g}{\log d_\varepsilon}+r_\varepsilon$.
We may then condition on all the possible values for $A_{\varepsilon}$ and write the left-hand side in the limit as \[\frac{C_{\varepsilon}^{\delta}}{\sqrt{-2\pi \log\varepsilon}}\int_{(-|\log\varepsilon|^{2/3},|\log\varepsilon|^{2/3})}\mathbb{E}\left[F(h_{L}^{\varepsilon}+\frac{x}{\log d_\varepsilon} g+r_\varepsilon)\nu_{h_{L}+\frac{x}{\log d_\varepsilon}g+r_\varepsilon}(\partial\mathbb{H})^{-\frac{2s}{\gamma}}\mathds{1}_{E_{\delta}^{\varepsilon}(x)}\right]\exp\left(-\frac{(x-s\log\frac{d_\varepsilon}{\varepsilon})^2}{2|\log \varepsilon|}\right)dx\]
\[+ \frac{C_{\varepsilon}^{\delta}}{\sqrt{-2\pi \log\varepsilon}}\int_{(|\log\varepsilon|^{2/3},\infty)}\mathbb{E}\left[F(h_{L}^{\varepsilon}+\frac{x}{\log d_\varepsilon}g+r_\varepsilon )\nu_{h_{L}^{\varepsilon}+\frac{x}{\log d_\varepsilon}g+r_\varepsilon}(\partial\mathbb{H})^{-\frac{2s}{\gamma}}\mathds{1}_{E_{\delta}^{\varepsilon}(x)}\right]\exp\left(-\frac{(x-s\log\frac{d_\varepsilon}{\varepsilon})^2}{2|\log \varepsilon|}\right)dx\]
where $\mathds{1}_{E_{\delta}^{\varepsilon}(x)}=\lbrace\nu_{h_{L}^{\varepsilon}+\frac x{\log d_\varepsilon} g}(\partial\mathbb{H})\in[e^{\gamma(-x-\delta)},e^{\gamma(-x+\delta)}]\rbrace$, and where the domain $(-\infty,\vert\log\varepsilon\vert^{\frac23})$ does not appear since we condition on the event $H^{\varepsilon}$. Let us now choose $C_{\varepsilon}^{\delta}$ to be given by $\frac{\sqrt{-2\pi \log\varepsilon}}{2\delta}$; we then claim that the integral over $(|\log\varepsilon|^{2/3},\infty)$ vanishes in the $\varepsilon\rightarrow 0$ limit, while the integral over $(-|\log\varepsilon|^{2/3},|\log\varepsilon|^{2/3})$ will converge to $\left[F(h_{L})\nu_{h_{L}}(\partial\mathbb{H})^{-\frac{2s}\gamma}\right]$. This would yield the desired result.

Let us start with the integral over $(|\log\varepsilon|^{2/3},\infty)$. On this domain the integrand can be uniformly bounded by
\begin{align}\label{Case1}
&\mathbb{E}\left[F(h_{L}^{\varepsilon}+\frac{x}{\log\varepsilon}g)\nu_{h_{L}^{\varepsilon}+\frac{x}{\log\varepsilon}g}(\partial\mathbb{H})^{-\frac{2s}\gamma}\mathds{1}_{E_{\delta}^{\varepsilon}(x)}\right]\leq ||F||_{\infty}\mathbb{E}\left[\nu_{h_{L}^{\varepsilon}}(\mathbb{D}\cap \partial\mathbb{H})^{-\frac{2s}\gamma}\right] \quad \text{if }s>0,
\end{align}
\begin{align*}
&\mathbb{E}\left[F(h_{L}^{\varepsilon}+\frac{x}{\log\varepsilon}g)\nu_{h^{\varepsilon}_{L}+\frac{x}{\log d_\varepsilon}g}(\partial\mathbb{H})^{-\frac{2s}\gamma}\mathds{1}_{E_{\delta}^{\varepsilon}(x)}\right]\leq||F||_{\infty}e^{2s (x+\delta)}\quad\text{otherwise.}
\end{align*}
Since the first expression is uniformly bounded in $\varepsilon$ by Lemma \ref{CIU}, the integral does indeed vanish.

We next turn to the integral over $(-|\log\varepsilon|^{2/3},|\log\varepsilon|^{2/3})$. Thanks to Lemma~\ref{CIU} below (with the deterministic sequence given by $x/\log \varepsilon$) we already know that the integrand converges pointwise, by which we mean that for any $x\in\mathbb{R}$, 
\[ \lim\limits_{\varepsilon\rightarrow0}
\mathds{1}_{x\in(-|\log\varepsilon|^{2/3},|\log\varepsilon|^{2/3})}\mathbb{E}\left[F(h_{L}^{\varepsilon}+\frac{x}{\log\varepsilon}g)
\nu_{h_{L}^{\varepsilon}+\frac{x}{\log\varepsilon}g}(\partial\mathbb{H})^{-\frac{2s}\gamma}\mathds{1}_{E_{\delta}^{\varepsilon}(x)}\right]\exp\left(-\frac{x^2}{2|\log \varepsilon|}\right)\]
\[=\mathbb{E}\left[F(h_{L})\nu_{h_{L}}(\partial\mathbb{H})^{-\frac{2s}\gamma}\mathds{1}_{E_{\delta}(x)}\right].\]
Assuming for a moment that the dominated convergence theorem can be applied, we can use Fubini theorem to write that
\begin{align*}
\int_{\mathbb{R}}\mathbb{E}\left[F(h_{L})\nu_{h_{L}}(\partial\mathbb{H})^{-\frac{2s}\gamma}\mathds{1}_{E_{\delta}(x)}\right]dx&=\mathbb{E}\left[F(h_{L})\nu_{h_{L}}(\partial\mathbb{H})^{-\frac{2s}\gamma}\int_{\mathbb{R}}\mathds{1}_{E_{\delta}(x)}dx\right]\\
&=2\delta\left[F(h_{L})\nu_{h_{L}}(\partial\mathbb{H})^{-\frac{2s}\gamma}\right]
\end{align*}
which would conclude the proof. Therefore all that is left to prove is that the dominated convergence theorem holds true for the integral over the domain $(-|\log\varepsilon|^{2/3},|\log\varepsilon|^{2/3})$.

Like before we distinguish between two cases depending on the sign of $s$; when $s$ is positive we can use the same reasoning as in Equation \eqref{Case1} to reduce the problem to proving that \[\lim\limits_{\varepsilon\rightarrow0} \int_{-\infty}^{+\infty} \mathbb{E}\left[\nu_{h_{L}^{\varepsilon}}(\partial\mathbb{H}\cap\mathbb{D})^{-\frac{2s}\gamma}\mathds{1}_{E_{\delta}^{\varepsilon}(x)}\right]dx=\int_{-\infty}^{+\infty} \lim\limits_{\varepsilon\rightarrow0} \mathbb{E}\left[\nu_{h_{L}^{\varepsilon}}(\partial\mathbb{H}\cap\mathbb{D})^{-\frac{2s}\gamma}\mathds{1}_{E_{\delta}^{\varepsilon}(x)}\right]dx.\] 
To see why this is indeed true note that by Fubini theorem (first and third equality) and Lemma \ref{CIU} (second and fourth equality) one has that
\begin{align*}
\lim\limits_{\varepsilon\rightarrow0} \int_{-\infty}^{+\infty} \mathbb{E}\left[\nu_{h_{L}^{\varepsilon}}(\partial\mathbb{H}\cap\mathbb{D})^{-\frac{2s}\gamma}\mathds{1}_{E_{\delta}^{\varepsilon}(x)}\right]dx&=\lim\limits_{\varepsilon\rightarrow0} \mathbb{E}\left[\nu_{h_{L}^{\varepsilon}}(\partial\mathbb{H}\cap\mathbb{D})^{-\frac{2s}\gamma}\left((Y_{\varepsilon}-X_{\varepsilon} +2\delta)\vee 0\right)\right]\\
&=2\delta\mathbb{E}\left[\nu_{h_{L}}(\partial\mathbb{H}\cap\mathbb{D})^{-\frac{2s}{\gamma}}\right]\\
&=\int_{-\infty}^{+\infty} \mathbb{E}\left[\nu_{h_{L}}(\partial\mathbb{H}\cap\mathbb{D})^{-\frac{2s}\gamma}\mathds{1}_{E_{\delta}(x)}\right]dx\\&=\int_{-\infty}^{+\infty} \lim\limits_{\varepsilon\rightarrow0}\mathbb{E}\left[\nu_{h_{L}^{\varepsilon}}(\partial\mathbb{H}\cap\mathbb{D})^{-\frac{2s}\gamma}\mathds{1}_{E_{\delta}^{\varepsilon}(x)}\right],
\end{align*}
where $Y_{\varepsilon}=\frac{1}{\gamma}\log \nu_{h_{L}^{\varepsilon}-|\log\varepsilon|^{-1/3}g(z)}(\partial\mathbb{H})$ and $X_{\varepsilon}=\frac{1}{\gamma}\log \nu_{h_{L}^{\varepsilon}+|\log\varepsilon|^{-1/3}g(z)}(\partial\mathbb{H})$. This shows the dominated convergence when $s$ is positive. Conversely if $s<0$ then the integrand is smaller than 
$||F||_{\infty}\mathbb{E}\left[\nu_{h_{L}^{\varepsilon}+|\log\varepsilon|^{-1/3} g(z)}(\partial\mathbb{H})^{-\frac{2s}\gamma}\mathds{1}_{E_{\delta}(x)}\right]$. We can proceed in the same way as in the case where $s\geq 0$, concluding the proof.
\end{proof}

\subsubsection{Convergence of the moments of the Liouville measure}To finish with, we justify the computations that we have made before by giving a convergence result for the moments of the boundary measure determined by $h_{L}$.
\begin{lemma}[Convergence in the $q^{th}$-moment of the boundary measure for the upper half-plane]\label{CIU}
Let $h_L$ be defined by Equation \eqref{GFF_semi_circle} and $a_{\varepsilon}$ be any deterministic sequence with limit $0$ as $\varepsilon$ goes to zero.
Then the $q$-moments of $\nu_{h_L+a_{\varepsilon}g}(\partial\mathbb{H})$ for $q<\frac{2}{\gamma}(Q-\gamma)$ converge to the ones of $\nu_{h_L}(\partial\mathbb{H})$ as $\varepsilon$ goes to $0$. Moreover for any $F$ non-negative bounded continuous (with the topology of weak convergence) functional over $\Hr^{-1}(\mathbb{H})$ we have 
\begin{equation}\label{eq:toprove}
\lim\limits_{\varepsilon\rightarrow 0}  \mathbb{E}\left[F(h_L+a_{\varepsilon}g) \nu_{h_L+a_{\varepsilon}g}(\partial\mathbb{H})^{-\frac {2s}\gamma}\right]=\mathbb{E}\left[F(h_L)\nu_{h_L}(\partial\mathbb{H})^{-\frac{2s}{\gamma}}\right].
\end{equation}
\end{lemma}
\begin{proof}
We first show that $h_L$ has a moment of order $q$ for $q<\frac{2}{\gamma}(Q-\gamma)$: this follows from the result \cite[Corollary 6.11]{HRV16} in the case of the disk with background measure $\rho_0$ the uniform one on the boundary, and by observing that $z\mapsto m_{\rho_0}(G(z,\cdot))-m_{l}(G(z,\cdot))$ is bounded, where $l$ is as before.
Pushing forward by $\psi(z)=\frac{z-i}{z+i}$ yields the result.

Now since on any compact of $\mathbb{H}$ the law of $h_L+a_{\varepsilon}g$ converges in total variation to that of $h_L$, we have that on any compact $\nu_{h_L+a_{\varepsilon}g}$ converges in total variation to $\nu_{h_L}$ as $\varepsilon\rightarrow0$ (in the sense of Radon measures). Therefore to show the following equality for $q<\frac{2}{\gamma}(Q-\gamma)$ (which implies both the convergence in the $q$-moments and Equation~\eqref{eq:toprove})
\begin{equation}\label{eq:toprove2}
\lim\limits_{\varepsilon\rightarrow 0}  \mathbb{E}\left[F(h_L+a_{\varepsilon}g) \nu_{h_L+a_{\varepsilon}g}(\partial\mathbb{H})^{q}\right]=\mathbb{E}\left[F(h_L)\nu_{h_L}(\partial\mathbb{H})^{q}\right]
\end{equation} it is enough to ensure that the expectation term
\begin{equation}\label{eq:bound}
\mathbb{E}\left[\nu_{h_L+a_{\varepsilon}g}(\partial\mathbb{H})^q\right]
\end{equation}
is uniformly bounded on $\varepsilon>0$ and that for any positive $q<\frac{2}{\gamma}(Q-\gamma)$,
\begin{equation}\label{eq:bound2}
\lim\limits_{R\rightarrow+\infty}\limsup\limits_{\varepsilon\rightarrow0}\mathbb{E}\left[\nu_{h_L+a_{\varepsilon}g}(\partial\mathbb{H}\setminus R\mathbb{D})^q\right]=0.
\end{equation}

In the case where the exponent $q$ is negative, the bound \[\mathbb{E}\left[\nu_{h_L+a_{\varepsilon}g}^q(\partial\mathbb{H})\right]\leq\mathbb{E}\left[\nu_{h_L+a_{\varepsilon}g}^q(\partial\mathbb{H}\cap\mathbb{D})\right]=\mathbb{E}\left[\nu_{h_L}^q(\partial\mathbb{H}\cap\mathbb{D})\right]\]
allows to conclude to bound uniformly Equation~\eqref{eq:bound}. 

Conversely if $0<q<\frac{2}{\gamma}(Q-\gamma)$ let us introduce for $r>1$, 
\[A_n:=(-re^n,-re^{n-1})\cup (re^{n-1}, re^n).
\]
We claim that for $\varepsilon$ small enough we have 
\begin{equation}\label{eq:estimate}
\mathbb{E}\left[\nu_{h_L+a_{\varepsilon}g}^q(A_n)\right]\leq Ce^{-nb}
\end{equation}
for some positive constants $C$ and $b$, and clearly this is enough to prove both Equations~\eqref{eq:bound} and~\eqref{eq:bound2}. Now to see why Equation~\eqref{eq:estimate} holds true, let $c>0$ be such that $q<\frac{2}{\gamma}(Q-\gamma-c)$ and $\varepsilon>0$ such that $a_{\varepsilon}<c$. We decompose $h_0=h_r+h_a$ between radial and angular parts, where $(h_r(e^{-t}))_{t\in\mathbb{R}}$ as same law as a two-sided Brownian Motion $(B_{2t})_{t\in\mathbb{R}}$, and $h_r$ and $h_a$ are independent (see \cite{DMS14} for details).
Then by definition of $\nu_h$ we have that
\begin{align*}
    \mathbb{E}\left[\nu_{h_L+a_{\varepsilon}g}^q(A_n)\right]&=\lim\limits_{\delta\rightarrow0}\mathbb{E}\left[\left(
\int_{A_n}\delta^{\gamma^2/4}e^{\frac{\gamma}{2}(h_r(z)+h_a^{\delta}(z) + \frac{\gamma}{2}(G(z,0)+G(z,1))+(3/2\gamma -Q)g(z) + a_{\varepsilon} g(z))}
dz\right)^q\right]\\
&\leq \tilde{C}\mathbb{E}\left[e^{-nq\frac{\gamma}{2}(Q-\gamma-c)+q\frac{\gamma}{2} s_n}\right]\lim\limits_{\delta\rightarrow0}\mathbb{E}\left[\left(
\int_{A_n}\frac{1}{|z|^{q\frac{\gamma}{2}}}\delta^{\gamma^2/4}e^{\frac{\gamma}{2} h_a^{\delta}(z)}
dz\right)^q\right]
\end{align*}
where $s_n=\underset{t\in[2(n-1),2n]}{\sup}B_t$ and $\tilde{C}$ absorbs $r^{-q\frac{\gamma}{2}}$, the constant order $\frac{\gamma}{2}(G(z,0)+ G(z,1))+\gamma g(z)$ and the difference between $\log \left(|z|\vee 1\right)$ and $\log |z|$. On the one hand, by the Markov property for the Brownian Motion,
\[\mathbb{E}\left[e^{-nq\frac{\gamma}{2}(Q-\gamma-c)+q\frac{\gamma}{2} s_n}\right]\leq C e^{-n\frac{\gamma^2}{4} q(\frac{2}{\gamma}(Q-\gamma-c)-q)}\]
where $\frac{\gamma^2}{4} q(\frac{2}{\gamma}(Q-\gamma-c)-q)$ is positive by assumption. On the other hand, by Kahane convexity inequality (see \cite[Theorem 2.1]{RV13} or \cite[Lemma 5.4]{DMS14} for details) we have that the q-th moments of $\nu_{h_L}(A_n)$ are bounded uniformly in $n$ for $q<\frac{4}{\gamma^2}$ (which occurs since $\frac{2}{\gamma}(Q-\gamma)<\frac{4}{\gamma^2}$) so we get \[
\mathbb{E}\left[\left(
\int_{A_n}\frac{1}{|z|^{\gamma Q}}\delta^{\gamma^2/2}e^{\gamma h_a^{\delta}(z)}
dz\right)^q\right]\leq c_q e^{-n\gamma q Q}.
\]
This allows us to provide the desired bound
\[
\mathbb{E}\left[\nu_{h_L+a_{\varepsilon}g}^q(A_n)\right]\leq Ce^{-b n}.
\]
Therefore we obtain that
\[\mathbb{E}\left[\nu_{h_L+a_{\varepsilon}g}^q(\partial\mathbb{H})\right]=\mathbb{E}\left[\left(\nu_{h_L}(\partial\mathbb{H}\cap r\mathbb{D})+\sum_{n\geq0} \nu_{h_L+a_{\varepsilon}g}(A_n)\right)^q\right]\]
is uniformly bounded for $\varepsilon$ small enough, and that
\[\lim\limits_{R\rightarrow \infty}\limsup\limits_{\varepsilon\rightarrow 0}  \mathbb{E}\left[\nu_{h_L+a_{\varepsilon}g}^q(\partial\mathbb{H}\setminus R\mathbb{D})\right]=\lim\limits_{n\rightarrow \infty}\limsup\limits_{\varepsilon\rightarrow 0} \mathbb{E}\left[\left(\sum_{k\geq n} \nu_{h_L+a_{\varepsilon}g}(A_k)\right)^q\right] = 0.\]

This allows us to conclude that for any $q<\frac{2}{\gamma}(Q-\gamma)$ Equation~\eqref{eq:toprove2} holds, that is
\[\lim\limits_{\varepsilon\rightarrow 0}  \mathbb{E}\left[F(h_L+a_{\varepsilon}g) \nu_{h_L+a_{\varepsilon}g}^q(\partial\mathbb{H})\right]=\mathbb{E}\left[F(h_L)\nu_{h_L}^q(\partial\mathbb{H})\right].\]
Eventually it suffices to notice that $-\frac{2s}{\gamma}=\frac{2Q-3\gamma}{\gamma}<\frac{2}{\gamma}(Q-\gamma)$ for $\gamma>0$.
\end{proof}

With the same proof, the analog result for the bulk measure remains true, provided that we have chosen a second singularity in $\mathbb{H}$, considered as exponent $\frac{Q-3\gamma/2}{\gamma}$ and $q<\frac{1}{\gamma}(Q-\gamma)$ (the factor $2$ accounts for the fact that we consider the bulk measure instead of the boundary measure).

In order to obtain a result that fits to the setting of the previous proposition, we state here a result that can be found in \cite{MS13}, in the case of the whole-plane GFF. Using the odd/even decomposition of the whole-plane GFF also provides the same result for the GFF in the upper half-plane with free boundary conditions. 
\begin{prop}
\hspace{0cm}\\
Let $(r_n)_{n\in\mathbb{N}}$ be a sequence tending to $+\infty$ as $n\rightarrow+\infty$ and set $D_n:=r_n\mathbb{D}\cap\mathbb{H}$. On $D_n$ let $h_n$ be a GFF with zero (resp. free) boundary conditions on $\partial D_n\cap\mathbb{H}$ (resp. $\partial D_n\cap\partial\mathbb{H}$). Also let $h$ be a GFF on $\mathbb{H}$ with free boundary conditions.

Then on any bounded subset $D$ the total variational distance between the law of $h_n$ and $h$ restricted to $D$ goes to zero as $n\rightarrow +\infty$, seen as distribution modulo additive constant.
\end{prop}
Thanks to this result, we can assume that in Lemma \ref{CIU} we were working with $h_L^{\varepsilon}$ that we have already introduced before. This justifies the fact that the setting of the proof of Proposition \ref{Approximation} still applied to Lemma \ref{CIU}.
\subsection{Approximation result for the unit boundary length quantum disk with three log-singularities}\label{subsection:final}
Combining the previous statements yields the following result, which provides us with a limiting construction for the unit boundary length quantum disk with three log-singularities thanks to a procedure very similar to the one obtained for the unit boundary length quantum disk with three marked points.
\begin{theorem}[Approximation result for the unit boundary length quantum disk with three log-singularities]\label{Approx_HRV}
\hspace{0cm}\\
Denote $D^{\varepsilon}:=\frac{1}{\sqrt{\varepsilon}}\mathbb{D}\cap\mathbb{H}$ and let $h_0^{\varepsilon}$ be a GFF on $D^{\varepsilon}$ with zero (resp. free) boundary conditions on $\partial D^{\varepsilon}\cap\mathbb H$ (resp. $\partial D^{\varepsilon}\cap\partial\mathbb H$). Sample $\hat{w}^{\varepsilon}$ on $\frac{1}{\sqrt{\varepsilon}}\mathbb{D}\cap\partial\mathbb{H}$ according to the Liouville boundary measure of the field 
\[h^{\varepsilon}=h_0^{\varepsilon}+\frac{1}{2}(2Q-\gamma)\log\varepsilon + \frac{\gamma}{2} G_{D^{\varepsilon}}(z,0)\] 
under the probability measure $\propto\nu_{h^{\varepsilon}}(\partial\mathbb{H}) d\mathbb{P}^{\varepsilon}$. Eventually conformally map the $3$-uple $(0,\hat w^\varepsilon,\infty)$ to $(0,1,\infty)$ with a M\"obius transform of $\mathbb{H}$ and denote by $\hat{h}^\varepsilon$ the field thus obtained, along with $\left(\mu_{\hat h^{\varepsilon}}, \nu_{\hat h^{\varepsilon}}\right)$ its Liouville Quantum Gravity measures. 

Then $(\mu_{\hat h^{\varepsilon}}, \nu_{\hat h^{\varepsilon}})$ conditioned on the event that $\nu_{\hat h^{\varepsilon}}(\partial\mathbb{H})\in[e^{-\gamma\delta},e^{\gamma\delta}]$ converges weakly in law to $(\mu_{HRV}^{UBL},\nu_{HRV}^{UBL})$ as $\varepsilon$ goes to zero and then $\delta$ goes to zero, where $(\mu_{HRV}^{UBL},\nu_{HRV}^{UBL})$ is the unit boundary length quantum disk described by Equation \eqref{UBL} with log-singularities $(\gamma,\gamma,\gamma)$ at $(0,1,\infty)$.
\end{theorem}
\begin{proof} Recall that the law of the unit boundary length quantum disk is given by Equation \eqref{UBL}:\[
\expect{F(\mu^{UBL}_{HRV},\nu^{UBL}_{HRV})}=
\frac{\mathbb{E}\left[F(\frac{\mu_{h_L}}{\nu_{h_L}(\partial \mathbb{D})^2}, \frac{\nu_{h_L}}{\nu_{h_L}(\partial \mathbb{D})})\nu_{h_L}(\partial\mathbb{D})^{-\frac{2s}{\gamma}}\right]}{\mathbb{E}\left[\nu_{h_L}(\partial \mathbb{D})^{-\frac{2s}{\gamma}}\right]}
\]
This law is described by a random variable of the form $G(h_L)$ under the probability measure $\propto \nu_{h_L}(\partial\mathbb{H})^{-\frac{2s}{\gamma}}$, where $G$ only depends on $h_L-\frac{2}{\gamma}\log \nu_{h_L}(\partial\mathbb{H})$. We can therefore apply the previous results to obtain that for any bounded, continuous functional on the space of Radon measures on $\mathbb{H}\times\partial\mathbb{H}$ we have 
\begin{align*}
&\lim\limits_{\delta\rightarrow 0} \lim\limits_{\varepsilon\rightarrow 0}  \mathbb{E}\left[F(\mu_{h^{\varepsilon}}, \nu_{h^{\varepsilon}})\vert E^{\varepsilon}_{\delta}(\partial\mathbb{H})\cap H^{\varepsilon} \right]\\
&=\lim\limits_{\delta\rightarrow 0} \lim\limits_{\varepsilon\rightarrow 0}  
\mathbb{E}\left[F(\frac{\mu_{h^{\varepsilon}}}{\nu_{h^{\varepsilon}}(\partial\mathbb{H})^2}, \frac{\nu_{h^{\varepsilon}}}{\nu_{h^{\varepsilon}}(\partial\mathbb{H})})\vert E^{\varepsilon}_{\delta}(\partial \mathbb{H})\cap H^{\varepsilon} \right]\quad\text{ since we condition on }\nu_{h^{\varepsilon}}(\partial\mathbb{H})\in\lbrace e^{-\gamma\delta},e^{\gamma\delta}\rbrace\\
&=\frac{\mathbb{E}\left[F(\frac{\mu_{h_L}}{\nu_{h_L}(\partial\mathbb{H})^2}, \frac{\nu_{h_L}}{\nu_{h_L}(\partial\mathbb{H})})\nu_{h_L}(\partial\mathbb{H})^{-\frac{2s}{\gamma}}\right]}{\mathbb{E}\left[\nu_{h_L}(\partial\mathbb{H})^{-\frac{2s}{\gamma}}\right]}\quad\text{according to Proposition \ref{Approximation} }\\
&=\mathbb{E}\left[F(\mu_{HRV}^{UBL}, \nu_{HRV}^{UBL})\right].
\end{align*}
But we have conditioned here on the event $E^{\varepsilon}_{\delta}(\partial\mathbb{H})\cap H^{\varepsilon}$, so to prove the result it remains to show that $\lim\limits_{\delta\rightarrow 0} \lim\limits_{\varepsilon\rightarrow 0}  \hat{\mathbb{P}}^{\varepsilon}(H^{\varepsilon}\vert E^{\varepsilon}_{\delta}(\partial\mathbb{H}))=1$ which is in the statement of Lemma \ref{LOCS2}.
This concludes the proof.
\end{proof}

We have described before two similar limiting procedures whose only difference between them was that in the DMS approach, we sampled from the law of $h^{\varepsilon}$ under the usual probability measure $\mathbb{P}^{\varepsilon}_\delta$, while in the HRV approach we sampled from the law of $h^{\varepsilon}$ under the weighted probability measure $d\hat{ \mathbb{P}}^{\varepsilon}_\delta \propto \nu_{h^{\varepsilon}}(\partial\mathbb{H})d\mathbb{P}^{\varepsilon}_\delta$. This difference becoming negligible in the limit, we can adapt the result of Lemma \ref{LOCS} to the HRV approach:
\begin{lemma}\label{LOCS2}
\hspace{0cm}\\
In the setting of Proposition \ref{Sampling1} we have the following estimates: \begin{itemize}
\item Conditionally on $E^{\varepsilon}_{\delta}(\partial\mathbb{H})$, $\lim\limits_{\varepsilon\rightarrow0}\frac{\log|\hat{w}^{\varepsilon}|}{|\log\varepsilon|^{2/3}}=0$ in law.
\item $\lim\limits_{\delta\rightarrow 0} \lim\limits_{\varepsilon\rightarrow 0}  \hat{\mathbb{P}}^{\varepsilon}(H^{\varepsilon}\vert E^{\varepsilon}_{\delta}(\partial\mathbb{H}))=1$, 
where $H^{\varepsilon}=\lbrace A_{\varepsilon}-s\log\varepsilon\geq - |\log\varepsilon|^{2/3}\rbrace$ like above.
\end{itemize}
\end{lemma}
\begin{proof}
Thanks to the conformal mapping $\psi^{-1}:z\mapsto i\pi-\log z$, we can work on $\mathcal{S}$ with the field $\mathfrak h^{\varepsilon}:=\mathfrak h_0^{\varepsilon}-(Q-\gamma)\mathrm{Re}(z)+(Q-\gamma)\log\varepsilon$, where the latter is defined like in Lemma \ref{LOCS}. By doing so the sampled point $\hat{w}^{\varepsilon}$ is sent to $i\pi-\log \hat{w}^{\varepsilon}$, whose real part is precisely $-\log |\hat{w}^{\varepsilon}|$, and the variable $A_\varepsilon$ has the law of the mean value of $\hat{\mathfrak h}^{\varepsilon}+\frac\gamma2G_{D^\varepsilon}(\psi(z),\hat{w}^{\varepsilon})$ on $\lbrace0\rbrace\times[0,i\pi]$, where $\hat{\mathfrak h}^{\varepsilon}$ is defined in a way similar to $\mathfrak h^{\varepsilon}$ but on the domain $\hat{\mathcal D}^{\varepsilon}:=(\log \sqrt{\varepsilon} |\hat{w}^{\varepsilon}|,\infty)\times [0,i\pi]$

The first point then follows from the first item of Lemma \ref{LOCS} since the total variation distance between the law of $i\pi-\log \hat{w}^{\varepsilon}$ and $\hat\omega^{\varepsilon}$ sampled according to $\hat{\mathfrak h}^{\varepsilon}$ in $\mathcal{S}$ goes to zero when both $\varepsilon$ and $\delta$ go to zero since we condition on the event $E^{\varepsilon}_{\delta}(\partial\mathbb{H})$.

For the second point, and using the explicit expression of the field $\hat{\mathfrak h}^{\varepsilon}$ and of the Green's kernel $G_{D^\varepsilon}$, we may apply an horizontal translation and use the first point to see that $A_\varepsilon$ has the law of the mean value of the field $\hat{\mathfrak h}^{\varepsilon}+s\log\varepsilon+\delta_\varepsilon$ on $\lbrace\log |\hat{w}^{\varepsilon}|\rbrace\times [0,i\pi]$, where $\delta_\varepsilon=o(|\log \varepsilon|^{2/3})$ with probability $1-o(1)$. In the end we can conclude since we are in the framework of Lemma \ref{LOCS}---apart from the fact that we work under the weighted probability measure $\hat{\mathbb{P}}^\varepsilon$, but since we condition on the event $E^{\varepsilon}_{\delta}(\partial\mathbb{H})$ the claim remains valid.
\end{proof}

\subsection{Equivalence between the two definitions}\label{conclusion}
In this last subsection we eventually show that the two definitions that have been given to describe the unit boundary length quantum disk actually coincide in the sense of Theorem \ref{Equ}. 

\begin{theorem}[Equivalence of the perpectives]\label{theorem_three}
Let $\mathbb{D}$ be the unit disk and $(z_1,z_2,z_3)$ be distinct points on its boundary.
Let $(\mu_{HRV}^{UBL},\nu_{HRV}^{UBL})$ be the unit boundary length quantum disk with log-singularities $(\gamma,\gamma,\gamma)$ located at $(z_1,z_2,z_3)$.
Likewise assume that $(\mu_{DMS}^{UBL},\nu_{DMS}^{UBL})$ is an instance of the unit boundary length quantum disk with three marked points embedded into $\mathbb{D}$ so that the marked points are $(z_1,z_2,z_3)$.

Then $(\mu_{HRV}^{UBL},\nu_{HRV}^{UBL})$ and $(\mu_{DMS}^{UBL},\nu_{DMS}^{UBL})$ have same law.
\end{theorem}
\begin{proof}
We have described in the last two sections two procedures giving in the limit the law of the unit boundary length quantum disk with three marked points (Proposition \ref{ADMS1bis}) and three log-singularities (Theorem \ref{Approx_HRV}). Nevertheless the total variation distance between these two procedures when conditioned on $E_{\varepsilon}^{\delta}(\partial\mathbb{H})$ goes to zero as $\delta$ goes to zero, so they must give in the limit the same law. Therefore the result is true when we have chosen $(z_1,z_2,z_3)$ to be precisely $(-1,-i,1)$.

However, since for any distinct points $(z_1,z_2,z_3)$ in the boundary of the disk we can find a conformal mapping $\varphi$ sending this $3$-uplet to $(1,-1,-i)$, and since the law of the unit boundary length quantum disk with three log-singularities and the law of the unit boundary quantum disk with three marked points are invariant under conformal mapping (or to be more precise, under the change of variable \eqref{CCD2}), the result obtained can be extended to the setting of Theorem \ref{Equ}.
\end{proof}

We note that without loss of generality one can show a result similar to that of Theorem \ref{conclusion} when considering a slightly different framework, in which we work with the unit boundary length quantum disk which has one marked point in the bulk and one point on the boundary instead of three points on the boundary. In other words we choose a different way of fixing an embedding for a quantum surface on the disk. One can check that all the steps used in the proof of Theorem \ref{conclusion} remain true under this framework.

\bibliographystyle{alpha}
\bibliography{biblio}

\end{document}